\documentclass[11pt,a4paper]{article}

\usepackage{amsmath,amssymb,amsthm}

\usepackage{fullpage}
\usepackage{cite}
\usepackage{chngcntr}
\usepackage{textcomp}

\usepackage{stmaryrd}  
\usepackage{titletoc}  
\usepackage{mathrsfs}  
\usepackage{enumitem}
\usepackage{float}
\usepackage{graphicx}

\usepackage{url}
\usepackage[normalem]{ulem}

\theoremstyle{definition}

\theoremstyle{plain}
\newtheorem{theorem}{Theorem}[section]
\newtheorem{proposition}[theorem]{Proposition}

\newtheorem{lemma}[theorem]{Lemma}

\theoremstyle{remark}

\usepackage{parskip}
\makeatletter
\def\thm@space@setup{%
  \thm@preskip=\parskip \thm@postskip=0pt
}
\makeatother

\numberwithin{equation}{section}

\DeclareMathOperator*\essinf{ess\,inf}
\DeclareMathOperator*\esssup{ess\,sup}

\begin{document}

\title{Multi-dimensional Optimal Trade Execution under Stochastic Resilience\thanks{Financial support through the TRCRC 190 {\it Rationality and competition; the economic performance of individuals and firms} and {\it d-fine GmbH} is gratefully acknowledged. We thank Paulwin Graewe for valuable comments and suggestions. This paper was finished while Horst was visiting the Risk Management Institute at the National University of Singapore. Grateful acknowledgement is made for hospitality.} }

\author{Ulrich Horst\thanks{Humboldt University Berlin, Department of Mathematics and School of Business and Economics, Unter den Linden 6, 10099 Berlin; email: horst@math.hu-berlin.de} and Xiaonyu Xia\thanks{Humboldt University Berlin, Department of Mathematics, Unter den Linden 6, 10099 Berlin; email: xiaxiaon@math.hu-berlin.de}}
\maketitle

\begin{abstract}
	We provide a general framework for analyzing linear-quadratic multi-dimensional portfolio liquidation problems with instantaneous and persistent price impact and stochastic resilience. We show that the value function can be described by a multi-dimensional non-monotone backward stochastic Riccati differential equations (BSRDE) with a singular terminal condition in one component. We prove the existence of a solution to the BSRDE system and characterise both the value function and the optimal strategy in terms of that solution. We prove that the solution to the liquidation problem can be approximated by the solutions to a sequence of unconstrained problems with increasing penalisation of open positions at the terminal time. Our proof is based on a much fine a priori estimate for the approximating BSRDE systems, from which we infer the convergence of the optimal trading strategies for the unconstrained models to an admissible liquidation strategy for the original problem.  
\end{abstract}



\section{Introduction and overview}

	Let $T \in (0,\infty)$ and let $(\Omega,\mathcal F,(\mathcal F_t)_{t\in[0,T]},\mathbb P)$ be filtered probability space that carries a one di\-men\-sion\-al standard Brownian motion  $W=(W_t)_{t\in[0,T]}$. We assume throughout that $(\mathcal F_t)_{t\in[0,T]}$ is the filtration generated by $W$ completed by all the null sets and that $\mathcal F = \mathcal F_T$. For a  Euclidean space $H$ we denote by $L^2_{\mathcal F}(0,T;H)$ and $L^\infty_{\mathcal F}(0,T;H)$
the Banach spaces of all $H$-valued, $\mathcal F_t$-adapted, square-integrable stochastic process $f$ on $[0,T]$, endowed with the norms $(\mathbb E\int_0^T |f(t)|^2\,dt)^{1/2}$, respectively $\esssup_{t,\omega}|f(t,\omega)|$. 

	For any $d \in \mathbb N$ we denote by $\mathcal S^d$ and $\mathcal S^d_{+}$ the Euclidean space of all  $d\times d$  symmetric, respectively nonnegative definite $d \times d$ matrices. For any two matrices $A, B$ from $\mathcal S^d$ we write $A>B$ and $A\geq B$ if $A-B$ is positive definite, respectively nonnegative definite.  In what follows all equations and inequalities are to be understood in the $\mathbb P$-a.s.~sense. 

For a given $d \in \mathbb N$ we consider the $d$-dimensional linear-quadratic stochastic control problem
\begin{equation} \label{control}
	\essinf_{\xi \in L^2_{\mathcal F}(0,T;\mathbb R^d)} \mathbb E\left[ \int_0^T\frac{1}{2}\xi(s)^\mathrm{T}\Lambda\xi(s)+Y(s)^\mathrm{T}\xi(s)+\frac{1}{2}X(s)^\mathrm{T}\Sigma(s) X(s)\,ds\,		\right]
\end{equation}
subject to the state dynamics
\begin{equation} \label{state-dynamics} 
\left\{\begin{aligned}
&X(t)= x - \int_0^t \xi(s)\,ds,  \quad t \in [0,T],\\
&X(T)= 0, \\
&Y(t)=y + \int_0^t\{-\rho(s)Y(s)+\gamma\xi(s)\}\,ds, \quad t \in [0,T].
\end{aligned}\right.
\end{equation}
Here, $\Lambda \in \mathcal S^d$ is positive definite, $\gamma=\textup{diag}(\gamma_i)$ is a positive definite diagonal matrix, and $\rho=\textup{diag}(\rho_i)$ and $\Sigma$ are progressively measurable essentially bounded $\mathcal S^d_+$-valued process:
\begin{equation}\label{matrices}
	0<\Lambda,\gamma\in \mathcal S^d; \rho, \Sigma \in L^\infty_{\mathcal F}(0,T;\mathcal S^d_+).
\end{equation}

Control problems of the above form arise in models of multi-dimensional optimal portfolio liquidation under market impact. In such models, $X(t) \in \mathbb R^d$ denotes the portfolio the investor needs to liquidate, and $\xi(t) \in \mathbb R^d$ denotes the rate at which the different stocks are traded at time $t \in [0,T]$. The terminal constraint $X(T) = 0$ is the {\it liquidation constraint}. The process $Y$ describes the persistent price impacts caused by past trades. 
The process $\rho$ describes how fast the order books recover from past trades. The matrix $\Lambda$ describes an additional instantaneous impact factor. The first two terms of running cost function in (\ref{control}) capture the expected liquidity cost resulting from the instantaneous and the persistent impact, respectively. The third term can be interpreted as a measure of the market risk associated with an open position. It penalises slow liquidation. 

The majority of the portfolio liquidation literature allows for only one of the two possible price impacts. The first approach, initiated by Bertsimas and Lo~\cite{BertsimasLo98} and Almgren and Chriss~\cite{AlmgrenChriss00} and later generalized by many authors including \cite{AnkirchnerJeanblancKruse14, GraeweHorstQiu13, KrusePopier16} describes the price impact as a purely temporary effect that depends only on the present trading rate. A second approach, initiated by Obizhaeva and Wang~\cite{ObizhaevaWang13} and later generalized in, e.g.~\cite{FruthSchoenebornUrusov14, FruthSchoenebornUrusov15} assumes that price impact is persistent with the impact of past trades on current prices decaying over time.  
Graewe and Horst \cite{GH2017} studied optimal execution problems with instantaneous and persistent price impact and stochastic resilience. In their model the value function can be represented by the solution to a coupled non-monotone three-dimensional stochastic Riccati equation. 

In this paper we extend the model in \cite{GH2017} to multi-dimensional portfolios. Our approach uses a penalisation method; we show that the solution to the liquidation problem can be obtained by a sequence of solutions to unconstrained problems where the terminal state constraint is replaced by an increased penalisation of open positions at the terminal time.  The penalisation technique has previously been applied to {\it one-dimensional} liquidation problems by many authors. When the value function can be characterised by a {\it one-dimensional} equation, the monotonicity of value function in the terminal condition can easily be established using standard comparison principles for PDEs or BSDEs. These comparison principles typically do not carry over to higher dimensions, which renders the analysis of multi-dimensional problems, especially the verification result much more complex. We extend the penalisation method to multi-asset liquidation problems with stochastic resilience. As a byproduct we obtain a convergence result for the single-asset model analysed in \cite{GH2017}. The convergence result provides an important consistency result for both the constrained and the unconstrained liquidation problem. Considering sequences of unconstrained problems with increasing penalization of open positions is reasonable only if some form of convergence for the optimal trading strategies can be established. Likewise, the constrained problem should approximate unconstrained problems in some sense.   

Several multi-dimensional liquidation models with deterministic cost functions and deterministic resilience have previously been considered in the literature. The  special case $\rho\equiv0, y=0, $ and $\Sigma\equiv const.$ in \eqref{control} and \eqref{state-dynamics} corresponds to the multiple asset model of  Almgren and Chriss \cite{AlmgrenChriss00}. This model was generalised by Kratz and Sch\"oneborn \cite{KratzSchoeneborn12} to discrete-time multi-asset liquidation problems when an investor trades simultaneously in a traditional venue and a dark pool. In the follow-up work \cite{KratzSchoeneborn15}, the same authors studied a continuous-time multi-asset liquidation problem with dark pools. The benchmark case of deterministic coefficients and zero permanent impact ($\gamma =0$) in (\ref{state-dynamics}) corresponds to the model in \cite{KratzSchoeneborn15} without a dark pool.  A model of optimal basket liquidation for a CARA investor with general deterministic cost function was  analyzed by Schied et al \cite{SchiedSchoenebornTehranchi10}. In their model there is no loss in generality in restricting the class of admissible liquidation strategies to deterministic ones. Later, Sch\"oneborn  \cite{Schoeneborn16} considered an infinite-horizon multi-asset portfolio liquidation problem for a general von Neumann-Morgenstern investor with general deterministic temporary and linear permanent impact functions. He characterized the value function as the solution to the two-dimensional PDEs and showed that the optimal portfolio process depends only on the co-variance and cross-asset market impact of the assets in his setting.  Alfonsi et al. \cite{AlfonsiKloeckSchied16} considered a discrete-time model of optimal basket liquidation with linear transient price impact and general deterministic resilience. Existence of an optimal liquidation strategy in their model is guaranteed if the decay kernel corresponds to a matrix-valued positive definite function but additional assumptions on the decay kernel are required for the optimal liquidation strategies to be well-behaved. In their recent paper, Schneider and Lillo \cite{SchneiderLillo2017} established necessary conditions on the size and impact of cross-impact for the absence of dynamic arbitrage in a continuous time version of \cite{AlfonsiKloeckSchied16} that can be directly verified on data.

We provide a general framework for analyzing multi-dimensional portfolio liquidation problems with linear-quadratic cost functions and stochastic resilience. If the cost and the resilience coefficients are described by general adapted stochastic processes, then the value functions can be described by a matrix-valued linear-quadratic backward stochastic Riccati differential equation (BSRDE) with a singular terminal component. In order to prove the existence of optimal liquidation strategies we first analyse the unconstrained problems with finite end costs. We show that the value functions for unconstrained problems are given by the solutions to BSRDE systems with finite terminal value after a linear transformation. For the benchmark case of uncorrelated assets this systems can be decomposed into a series of subsystems for which a priori estimates similar to those in \cite{GH2017} can be established. Using a novel comparison result for matrix-valued BSRDEs given in the appendix, we prove that the solutions to the BSRDE systems can be uniformly bounded from above and below on compact time intervals by two benchmark models with uncorrelated assets. This allows us to prove that the pointwise (in time) limit of the solutions to these systems exists when the degree of penalisation tends to infinity. This limit yields a candidate value function for the liquidation problem. 

The verification argument is much more involved. It requires a much finer then usual a priori estimate for the approximating BSRDE systems,  from which to infer the convergence of the optimal trading strategies for the unconstrained models to an admissible liquidation strategy for the original problem. The convergence of the optimal strategies allows us to carry out the verification argument and to prove that the limiting BSRDE does indeed characterise the value function of the original liquidation problem. We emphasise that in our multi-dimensional setting the convergence of the optimal trading strategies is required for the verification argument. This is typically not the case in one-dimensional models where much coarser a priori estimates are sufficient to carry out the verification argument.     

When all the cost coefficients in our model are deterministic constants, then the value function associated with the optimization problem above can be described by a linear-quadratic ODE system with singular terminal component. The deterministic benchmark case can be analyzed numerically. Our simulations suggest that the correlation between the assets' fundamental value is a key determinant of optimal liquidation strategies. Consistent with the \cite{GH2017} we also find that when the instantaneous impact factor is {\sl low}, then the optimal liquidation strategies are strongly convex with the degree of convexity depending on the correlation and the relative liquidity across assets. Surprisingly, a qualitatively similar result is obtained for {\sl high} persistent impact factors. The intuitive reason is that early trading benefits from resilience.  

The remainder of this paper is structured as follows. The model and main results are summarised in Section \ref{model}. All proofs are carried out in Section \ref{Proofs of the main results}. Numerical simulations are provided in Section 4. Section 5 concludes. An  appendix contains a multi-dimensional comparison principle for BSRDEs and technical estimates that are omitted in earlier sections.

{\it Notational conventions.}
We adopt the convention that $C$ is a constant, which may vary from line to line. Moreover, we will use the following spaces of progressively measurable processes:
\begin{align*}
	\mathcal L^2_{\mathcal F}(\Omega;C([0,T];H)) =& \left\{ f: \Omega \to C([0,T];H) : \mathbb E [\max_{t\in[0,T]}|f(t)|^2] < \infty\right\} \\
	\mathcal L^\infty_{\mathcal F}(\Omega;C([0,T];H)) =& \left\{ f: \Omega \to C([0,T];H) :  \esssup_{\omega\in \Omega} \max_{t\in[0,T]}|f(t,\omega)| < \infty \right\}.
\end{align*}

We say that a sequence of stochastic processes $\{f^n(\cdot)\}_{n \in \mathbb N}$ converges compactly to $f(\cdot)$ on $[0,T)$ if it converges uniformly to $f(\cdot)$ on every compact subinterval. Whenever the notation $T^-$ appears  we mean that the statement holds for all the $T'<T$ when $T^-$ is replaced by $T'$, e.g. $L^2_{\mathcal F}(0,T^-;\mathcal S^d)=\bigcap_{T'<T}L^2_{\mathcal F}(0,T';\mathcal S^d)$. 
For $f\in L_{\mathcal F}^\infty(\Omega,C([0,T^-];H))$ we mean by $L^\infty$-$\lim_{t\rightarrow T}|f_t|= \infty$ that for every $C>0$ there exists $T'<T$ such that $f_t\geq C$ for all $t\in[T',T)$, $\mathbb P$-a.e. 

For any vector or matrix $B=(b_{ij}),$ we put $|B|:=\sqrt{\sum_{ij}b^2_{ij}}$. For $A \in {\cal S}^d$ the largest (smallest) eigenvalue is denoted $a_{\max}$ ($a_{\min}$) and $|A|_{2,2}=a_{\max}$ denotes the induced matrix norm. Then, $|A|_{2,2}\leq |A|$. For any $A \in {\cal S}^d_+$, the square root $\sqrt{A}$ exists and $AB$ and $\sqrt{A} B \sqrt{A}$ have same eigenvalues. Hence $|AB|_{2,2} = |\sqrt{A} B \sqrt{A}|_{2,2}$ and 
$\textup{tr}(AB)=\textup{tr}(\sqrt{A} B \sqrt{A}).$ Moreover, $a_{\min} I_d \leq A \leq a_{\max} I_d$ where $I_d$ denotes the $d \times d$ identity matrix.


\section{Main Results}\label{model}

For any initial state $(t,x,y) \in [0,T) \times \mathbb R^d \times \mathbb R^d$ we define by 
\begin{equation} \label{control-problem}
	V(t,x,y):=\essinf_{\xi \in \mathcal A(t,x)} \mathbb E\left[ \int_t^T\frac{1}{2}\xi(s)^\mathrm{T}\Lambda\xi(s)+Y(s)^\mathrm{T}\xi(s)+\frac{1}{2}X(s)^\mathrm{T}\Sigma(s) X(s)\,ds\, | \mathcal F_t			\right]
\end{equation}
the {\it value function} of the stochastic control problem \eqref{control} subject to the state dynamics
\begin{equation} \label{state-dynamics2} 
\left\{\begin{aligned}
&dX(s)= - \xi(s)\,ds,  \quad s \in [t,T],\\
&dY(s)=\{-\rho(s)Y(s)+\gamma\xi(s)\}\,ds, \quad s \in [t,T].
\end{aligned}\right.
\end{equation}
Here, the essential infimum is taken over the class $\mathcal A(t,x)$ of all admissible {\it liquidation strategies}, that is over  all {\it trading strategies} $\xi\in L_{\mathcal F}^2(t,T;\mathbb R^d)= L_{\mathcal F}^2$ that satisfy the liquidation constraint 
\[
	X(T)=0 \quad \mbox{a.s.}
\]

We characterise the value function by a unique solution to a matrix-valued backward BSRDE with singular terminal condition. Our approach is based on an approximation argument. To this end, we consider, for any $n \in \mathbb N$, the value function 
\begin{equation} \label{control-modified}
\begin{split}
V^n(t,x,y)= \essinf_{\xi \in L_{\mathcal F}^2}\mathbb E\bigg\{&nX(T)^\mathrm{T}X(T)+2Y(T)^\mathrm{T}X(T)\\
&+\int_t^T\big[\frac{1}{2}\xi(s)^\mathrm{T}\Lambda\xi(s)+Y(s)^\mathrm{T}\xi(s)+\frac{1}{2}X(s)^\mathrm{T}\Sigma(s) X(s)\big]\,dr|\mathcal F_t\bigg\}
\end{split}
\end{equation}
of a corresponding unconstrained optimisation problem where the liquidation constraint is replaced by a finite penalty term.
We solve the unconstrained problem first and then show that the solutions to \eqref{control-modified} converge to the value function \eqref{control-problem} as $n\rightarrow \infty.$

	A pair of random fields $(V^n,N^n):\Omega\times[0,T)\times \mathbb R^d\times\mathbb R^d\rightarrow \mathbb R\times\mathbb R$ is called a \textit{classical solution} to \eqref{control-modified} if it satisfies the following conditions:
\hspace{-3mm}\begin{itemize}
	\item for each $t\in[0,T)$, $V^n(t,x,y)$ is continuously differentiable in $x$ and $y$,
	\item for each $(x,y)\in\mathbb R^d\times \mathbb R^d$, $(V^n(t,x,y),\partial_x V^n(t,x,y)),\partial_y V^n(t,x,y))_{t\in[0,T]}\in L_{\mathcal F}^\infty(\Omega;C([0,T];\mathbb R\times \mathbb R^d\times \mathbb R^d))$,
	\item for each $(x,y)\in\mathbb R^d\times \mathbb R^d$, $(N^n(t,x,y))_{t\in[0,T]}\in L^2_{\mathcal F}(0,T;\mathbb R)$,
	\item for all $0\leq t\leq s\leq T$ and $x,y\in\mathbb R^d$ it holds that
\begin{equation}\label{HJB}
\hspace{-4.5mm} \left\{
\begin{aligned}
&\begin{aligned}
V^n(t,x,y) =  &V^n(s,x,y)+\int_t^s\inf_{\xi\in\mathbb R^d} \left\{
-\partial_x V^n(r,x,y)^\mathrm{T}\xi -\partial_y V^n(r,x,y)^\mathrm{T}(\rho(r) y -\gamma\xi) \right. \\
	 & \left. \quad +\frac{1}{2}\xi^\mathrm{T}\Lambda\xi+y^\mathrm{T}\xi+\frac{1}{2}x^\mathrm{T}\Sigma(r) x \right\}\,dr 
	  -\int_t^sN^n(r,x,y)\,dW(r),
 \end{aligned}\\
& V^n(T,x,y)=nx^\mathrm{T}x+2y^\mathrm{T}x.
\end{aligned}
\right.
\end{equation}
\end{itemize}
The linear-quadratic structure of the control problem suggest the ansatz 
\begin{equation} \label{LQ-ansatz} 
\begin{split}
V^n(t,x,y) &=\begin{bmatrix}
 x^\mathrm{T}&y^\mathrm{T}
\end{bmatrix}P^n(t)\begin{bmatrix}
\ x\ \\
y 
\end{bmatrix}\\
N^n(t,x,y) &=\begin{bmatrix}
 x^\mathrm{T}&y^\mathrm{T}
\end{bmatrix}M^n(t)\begin{bmatrix}
\ x\ \\
y 
\end{bmatrix}
\end{split}
\end{equation}
for the solution to the HJB equation, where $P^n, M^n$ are progressively measurable $\mathcal S^{2d}$-valued processes. The following lemma shows that this ansatz reduces our HJB equation \eqref{HJB} to the matrix-valued stochastic Riccati equation,
\begin{equation}  \label{BSRDE} 
\begin{aligned}
-dP(t)=&\Bigg\{-\left(P(t)\begin{bmatrix}
-I_d\\
\gamma
\end{bmatrix}+\begin{bmatrix}
0\\
I_d
\end{bmatrix}\right)\Lambda^{-1}\left(\begin{bmatrix}
-I_d&\gamma
\end{bmatrix}P(t)+\begin{bmatrix}
0&I_d
\end{bmatrix}\right)\\
&\quad+P(t)\begin{bmatrix}
0&0\\
0&-\rho(t)
\end{bmatrix}+\begin{bmatrix}
0&0\\
0&-\rho(t)
\end{bmatrix}P(t)+\begin{bmatrix}
\Sigma(t)&0\\
0&0
\end{bmatrix}
\Bigg\}\,dt-M(t)\,dW(t),
\\\
P(T)=&\begin{bmatrix}
nI_d&I_d\\
I_d&0
\end{bmatrix},
\end{aligned}
\end{equation}
where $I_d$ is the $d\times d$ identity matrix. The proof is standard and hence omitted. 
\begin{lemma}\label{feedback}
	If the vector $$(P^n,M^n)\in L_{\mathcal F}^\infty(\Omega;C([0,T];\mathcal S^{2d}))\times L_{\mathcal F}^2(0,T;\mathcal S^{2d})$$ solves the BSRDE system~\eqref{BSRDE}, then the random field $(V^n,N^n)$ given by the linear-quadratic ansatz \eqref{LQ-ansatz} solves the HJB equation~\eqref{HJB} and the infimum in~\eqref{HJB} is attained by
\begin{equation} \label{feedback-form}
	\xi^{n,*}(t,x,y)=-\Lambda^{-1}\left(\begin{bmatrix}
-I_d&\gamma
\end{bmatrix}P^n(t)+\begin{bmatrix}
0&I_d
\end{bmatrix}\right)\begin{bmatrix}
\ x\ \\
y 
\end{bmatrix}.
\end{equation}
\end{lemma}

Bismut \cite{Bismut1976} and Peng \cite{Peng1992} proved the existence and uniqueness of solutions to general BSRDEs, but require the coefficient of degree zero to be nonnegative definite. This requirement is not satisfied in our case. We can overcome this problem by the linear transformation $$Q= P+\begin{bmatrix}
0&0\\
0&\gamma^{-1}
\end{bmatrix}$$ and by considering the resulting BSRDE:
\begin{equation}  \label{BSRDE'} 
\begin{aligned}
-dQ(t)=&\Bigg\{-Q(t)\begin{bmatrix}
-I_d\\
\gamma
\end{bmatrix}\Lambda^{-1}\begin{bmatrix}
-I_d&\gamma
\end{bmatrix}Q(t)+Q(t)\begin{bmatrix}
0&0\\
0&-\rho(t)
\end{bmatrix}+\begin{bmatrix}
0&0\\
0&-\rho(t)
\end{bmatrix}Q(t)\\
&\quad+\begin{bmatrix}
\Sigma(t)&0\\
0&\gamma^{-1}\rho(t)+\rho(t)\gamma^{-1}
\end{bmatrix}
\Bigg\}\,dt-M(t)\,dW(t),
\\
Q(T)=&\begin{bmatrix}
nI_d&I_d\\
I_d&\gamma^{-1}
\end{bmatrix}.
\end{aligned}
\end{equation}
The matrix $\begin{bmatrix}
nI_d&I_d\\
I_d&\gamma^{-1}
\end{bmatrix}$ is nonnegative definite if $n\geq\gamma_{\max}.$ In this case, all the coefficients in \eqref{BSRDE'} satisfy the requirements in \cite{Bismut1976} and \cite{Peng1992} (see also \cite[Proposition 2.1]{Kohlmann2003}). Hence, we have the following existence result.
\begin{theorem}
For every $n\geq\gamma_{\max},$ the BSRDE \eqref{BSRDE'} has a unique solution  $$(Q^n,M^n)\in L_{\mathcal F}^\infty(\Omega;C([0,T];\mathcal S^{2d}_+))\times L_{\mathcal F}^2(0,T;\mathcal S^{2d}).$$
\end{theorem}
The preceding theorem implies the existence and the uniqueness of a solution 
$$(P^n,M^n)\in L_{\mathcal F}^\infty(\Omega;C([0,T];\mathcal S^{2d}))\times L_{\mathcal F}^2(0,T;\mathcal S^{2d})$$
to the BSRDE \eqref{BSRDE}. The following theorem shows that the solution to the unconstrained optimisation problem can be given in terms of $P^n$. The proof is given in Section 3 below.

\begin{theorem} \label{OPT-modified}
Let $n>n_0,$ where
 \begin{equation}\label{n_0}
\begin{aligned}
n_0:&=\max\{\lambda_{\min}(\sqrt{1+\alpha}+1)+\gamma_{\min},(\beta+1)\gamma_{\max}+1\}, \\
\beta:&=3+2||\rho||^2_{L^{\infty}} \\
\alpha: & =\frac{||\Sigma||_{L^{\infty}}+2\gamma_{max}||\rho||_{L^{\infty}}}{\lambda_{\min}}\\
\end{aligned}
\end{equation}
Let $(P^n,M^n)$ be the unique solution of the BSRDE \eqref{BSRDE}. Then the value function  \eqref{control-modified} is of the linear-quadratic form 
\begin{equation*}
\begin{split}
V^n(t,x,y) &=\begin{bmatrix}
 x^\mathrm{T}&y^\mathrm{T}
\end{bmatrix}P^n(t)\begin{bmatrix}
\ x\ \\
y 
\end{bmatrix}
\end{split}
\end{equation*}
 and the optimal $\xi^{n,*}$ is given in feedback form by \eqref{feedback-form}.
\end{theorem}
Intuitively, the solution to \eqref{control-problem} should be the limit of the solutions to \eqref{control-modified} as $n\rightarrow \infty$. The following two theorems show that this limit is well-defined and characterises the value function of our liquidation problem. The proofs are given in Section 3 below. 
\begin{theorem}\label{existence}
For any $t\in [0,T),$ the limit
$$Q(t):=\lim\limits_{n\rightarrow +\infty}Q^n(t)$$ exists and $Q^n(\cdot)$ converges compactly to $Q(\cdot)$ on $[0,T).$ Moreover,  there exists $M\in L_{\mathcal F}^2(0,T^-;\mathcal S^{2d})$ such that $(Q,M)$ solves the equation
\begin{equation}\label{limitequation} 
\begin{aligned}
&
\begin{aligned}
-dQ(t)=\Bigg\{&-Q(t)\begin{bmatrix}
-I_d\\
\gamma
\end{bmatrix}\Lambda^{-1}\begin{bmatrix}
-I_d&\gamma
\end{bmatrix}Q(t)+Q(t)\begin{bmatrix}
0&0\\
0&-\rho(t)
\end{bmatrix}+\begin{bmatrix}
0&0\\
0&-\rho(t)
\end{bmatrix}Q(t)\\
&+\begin{bmatrix}
\Sigma(t)&0\\
0&\gamma^{-1}\rho(t)+\rho(t)\gamma^{-1}
\end{bmatrix}
\Bigg\}\,dt-M(t)\,dW(t).
\end{aligned}
\end{aligned}
\end{equation}
on $[0,T).$ Furthermore, $$\liminf\limits_{t\rightarrow T}|Q(t)|=+\infty.$$
\end{theorem}
By Theorem \ref{existence} we also obtain the existence of the limit of the optimal strategies as $n \to \infty:$
$$\xi^*(t,x,y):=\lim\limits_{n\rightarrow \infty}\xi^{n,*}(t,x,y)=-\Lambda^{-1}\begin{bmatrix}
-I_d&\gamma
\end{bmatrix}Q(t)\begin{bmatrix}
\ x\ \\
y 
\end{bmatrix}.$$
This allows us to state the main result of this paper. 
\begin{theorem}\label{OPT}
Let $Q$ be the limit given in Theorem \ref{existence} and put
$P=Q-\begin{bmatrix}
0&0\\
0&\gamma^{-1}
\end{bmatrix}.$
Then the value function \eqref{control-problem} is given by 
\begin{equation} 
\begin{split}
V(t,x,y) &=\begin{bmatrix}
 x^\mathrm{T}&y^\mathrm{T}
\end{bmatrix}P(t)\begin{bmatrix}
\ x\ \\
y 
\end{bmatrix}
\end{split}
\end{equation}
 and  the optimal control is given by 
 \begin{equation}
 \xi^*(t,x,y) =-\Lambda^{-1}\begin{bmatrix}
-I_d&\gamma
\end{bmatrix}Q(t)\begin{bmatrix}
\ x\ \\
y 
\end{bmatrix}.
\end{equation}
\end{theorem}


\section{Proofs}\label{Proofs of the main results}
In this section, we give the proofs of the Theorems \ref{OPT-modified}-\ref{OPT}. 
In a first step, we bound (with respect to the partial order on the cone of positive semi-definite matrices) the processes $Q^n$ by a matrix-valued processes whose limiting behaviour at the terminal  time can be inferred from a one-dimensional benchmark model (Lemma \ref{rough estimate}). This will enable us to prove the existence of  the limit $\lim\limits_{n\rightarrow\infty}Q^n$ (Theorem \ref{existence}). In a second step, we establish upper and lower bounds for $$\sqrt{\Lambda^{-1}}\begin{bmatrix}
-I_d&\gamma
\end{bmatrix}Q^n\begin{bmatrix}
-I_d\\
\gamma
\end{bmatrix}\sqrt{\Lambda^{-1}}$$ near the terminal time (Proposition \ref{simple estimate}), from which we will infer the convergence of the strategies $\{\xi^{n,*}\}$ to an admissible liquidation strategy. 

{\bf{Notation.}}
The following notion will be useful. For a generic matrix $Q\in S^{2d}$, we write
\begin{equation} \label{def-F}
Q_{2d\times 2d}=\begin{bmatrix}
A_{d\times d}&B_{d\times d}\\
B^\mathrm{T}_{d\times d}&C_{d\times d}\\
\end{bmatrix},
\end{equation}
and set $D:=(A- \gamma B^\mathrm T)$, $E:=(\gamma C-B)$ and $F:=D+E\gamma$.


\subsection{A priori estimates}\label{A Priori Estimates}
If $d=1,$ then $Q=\begin{bmatrix} A & B\\ B & C\\ \end{bmatrix}$ and the system \eqref{BSRDE'} simplifies to the three-dimensional BSRDE: 
\begin{equation}  \label{3dBSDE} 
\left\{
\begin{aligned}
-dA(t)&=\left\{\sigma(t)-\lambda^{-1} (A(t)-\gamma B(t))^2\right\}dt-M^A(t)\,dW(t)\\
-dB(t)&=\left\{-\rho(t) B(t)+\lambda^{-1}(\gamma C(t)-B(t))(A(t)-\gamma B(t))\right\}dt-M^B(t)\,dW(t)\\
-dC(t)&= \left\{-2\rho(t) C(t)+2\rho(t)\gamma^{-1}-\lambda^{-1}(\gamma C(t)-B(t))^2\right\}dt-M^C(t)\,dW(t)\\
A(T)&=n, B(T)=1, C(T)=\gamma^{-1}.
\end{aligned}
\right.
\end{equation}
Analogous to the a priori estimates in \cite{GH2017},  we have the following bounds on $[0,T]$:
\begin{equation*}
\begin{aligned}
	\underline D(t):=\frac{\gamma}{e^{\lambda^{-1}\gamma (T-t)}(1+\frac{\gamma}{n-\gamma})-1}&\leq D(t)\leq \lambda\kappa \coth\left(\kappa(T-t)+\textnormal{arccoth}\frac{\lambda^{-1}(n-\gamma)}{\kappa}\right)=:\overline D(t), \\
	\underline B(t):=e^{-\|\rho\|_{L^\infty}(T-t)}&\leq B(t)\leq 1,\\
	0&\leq E(t)\leq 1,\\
	0&\leq C(t)\leq \gamma^{-1},
\end{aligned}
\end{equation*}
where $\kappa:=\sqrt{2\lambda^{-1}\max\{\|\sigma\|_{L^\infty},\gamma\|\rho\|_{L^\infty}\}}.$ Thus, 
\begin{equation} \label{priori}
\begin{aligned}
	\frac{\gamma}{e^{\lambda^{-1}\gamma (T-t)}(1+\frac{\gamma}{n-\gamma})-1}&\leq A(t), F(t)\leq \lambda \kappa \coth\left(\kappa(T-t)+\textnormal{arccoth}\frac{\lambda^{-1}(n-\gamma)}{\kappa}\right)+\gamma, \\
	\gamma^{-1}e^{-\|\rho\|_{L^\infty}(T-t)}&\leq C(t)\leq  \gamma^{-1}.
\end{aligned}
\end{equation}

\subsubsection{A first (rough) estimate}

If $\Lambda, \Sigma$ were diagonal matrices, the  BSRDE system \eqref{BSRDE'} would separate into $d$ subsystems, which are similar to the three-dimensional system \eqref{3dBSDE}. Our idea is thus to first bound $(\Lambda, \Sigma)$ from above and below by diagonal matrices $(\lambda_{\max} I_d, |\Sigma(t)| I_d)$ and $ (\lambda_{\min} I_d, 0)$, respectively, and then to prove that the solutions to the resulting BSRDEs provide upper and lower bounds for the processes $Q^n$. 

More precisely, we consider the following BSRDEs:
 \begin{equation}  \label{BSRDEmax}
\begin{aligned}
-dQ(t)=&\Bigg\{-Q(t)\begin{bmatrix}
-I_d\\
\gamma
\end{bmatrix}\lambda^{-1}_{\max}I_d\begin{bmatrix}
-I_d&\gamma
\end{bmatrix}Q(t)+Q(t)\begin{bmatrix}
0&0\\
0&-\rho(t)
\end{bmatrix}+\begin{bmatrix}
0&0\\
0&-\rho(t)
\end{bmatrix}Q(t)\\
&\quad+\begin{bmatrix}
|\Sigma(t)| I_d&0\\
0&\gamma^{-1}\rho(t)+\rho(t)\gamma^{-1}
\end{bmatrix}
\Bigg\}\,dt-M(t)\,dW_t,
\\
Q(T)=&\begin{bmatrix}
nI_d&I_d\\
I_d&\gamma^{-1}
\end{bmatrix}
\end{aligned}
\end{equation}
and
\begin{equation}\label{BSRDEmin}  
\begin{aligned}
-dQ(t)=&\Bigg\{-Q(t)\begin{bmatrix}
-I_d\\
\gamma
\end{bmatrix}\lambda_{\min}^{-1}I_d\begin{bmatrix}
-I_d&\gamma
\end{bmatrix}Q(t)+Q(t)\begin{bmatrix}
0&0\\
0&-\rho(t)
\end{bmatrix}+\begin{bmatrix}
0&0\\
0&-\rho(t)
\end{bmatrix}Q(t)\\
&\quad+\begin{bmatrix}
0&0\\
0&\gamma^{-1}\rho(t)+\rho(t)\gamma^{-1}
\end{bmatrix}
\Bigg\}\,dt-M(t)\,dW_t,
\\
Q(T)=&\begin{bmatrix}
nI_d&I_d\\
I_d&\gamma^{-1}
\end{bmatrix}.
\end{aligned}
\end{equation}
Their solutions are denoted by $(Q^n_{\max},M^n_{\max})$ and $(Q^n_{\min},M^n_{\min}),$ respectively. The matrices  $Q^n_{\max}, Q^n_{\min}$ are of the form $$\begin{bmatrix}
A_1& & &B_1& & \\
 &\ddots & & &\ddots &\\
 &  &A_d &  & &B_d\\
 B_1& & & C_1& & \\
 &\ddots & & &\ddots &\\
 &  &B_d &  & &C_d\\
\end{bmatrix},$$ where each triple $(A_i, B_i, C_i)$ solves the BSRDE \eqref{3dBSDE} if $(\lambda, \gamma, \sigma,\rho)$ is replaced by $(\lambda_{\max}, \gamma_i, |\Sigma| ,\rho_i)$ and $(\lambda_{\min}, \gamma_i, 0,\rho_i),$  respectively. Moreover, from our comparison theorem given in the appendix [Theorem \ref{comparison}], we conclude that  the processes $Q^n$, $Q^n_{\max}$ and $Q^n_{\min}$ are nondecreasing in $n$ and
\begin{equation}\label{bounds}
Q^n_{\min}(t)\leq Q^n(t)\leq Q^n_{\max}(t)\quad \text{for all } t\in [0,T].
\end{equation}
Combing this inequality with the a priori estimates \eqref{priori}, we obtain the following result.

\begin{lemma}\label{rough estimate}
For every $n\geq\gamma_{max},$ the following a priori estimates hold for all $t\in[0,T]$:
\begin{equation} \label{generalestimate}
\begin{aligned}
	&\textup{diag}(\underline A^n_i)\leq A^n_{\min,i}\leq A^n\leq A^n_{\max}\leq \textup{diag}(\overline A_i), \\
	&\textup{diag}(\underline C^n_i)\leq C^n_{\min,i}\leq C^n\leq C^n_{\max}\leq  \gamma^{-1},\\
	&\textup{diag}(\underline F^n_i)\leq F^n_{\min,i}\leq F^n\leq  F^n_{\max}\leq \textup{diag}(\overline F_i).
\end{aligned}
\end{equation}
and $B^n_{\max}\leq I_d$
where 
\begin{equation*} 
\begin{aligned}
	&\underline C^n_i=\gamma_i^{-1}e^{-\|\rho_i\|_{L^\infty}(T-t)},\\
	&\underline A^n_i=\underline F^n_i=\frac{\gamma_i}{e^{\lambda_{\min}^{-1}\gamma_i (T-t)}(1+\frac{\gamma_i}{n-\gamma_i})-1},\\
	&\overline A_i=\overline F_i=\lambda_{\max}\kappa_i \coth\left(\kappa_i(T-t)\right)+\gamma_i, \\
	&\kappa_i=\sqrt{2\lambda_{\max}^{-1}\max\{\|\Sigma\|_{L^\infty},\gamma_i\|\rho_i\|_{L^\infty}\}}.
\end{aligned}
\end{equation*}
\end{lemma}

\subsubsection{A second (finer) estimate}

We are now going to bound the processes $$\sqrt{\Lambda^{-1}}\begin{bmatrix}
-I_d&\gamma
\end{bmatrix}Q^n\begin{bmatrix}
-I_d\\
\gamma
\end{bmatrix}\sqrt{\Lambda^{-1}}=\sqrt{\Lambda^{-1}}F^n\sqrt{\Lambda^{-1}}.$$ 
Multiplying $\begin{bmatrix}
-I_d&\gamma
\end{bmatrix}$ on the left and $\begin{bmatrix}
-I_d\\
\gamma
\end{bmatrix}$ on the right in \eqref{BSRDE'},  we see that $F^n$ satisfies
\begin{equation} \label{BSRDE_Fn}
\begin{aligned}
-dF^n(t)=&\Bigg\{\begin{bmatrix}
-I_d&\gamma
\end{bmatrix}\left(Q^n(t)\begin{bmatrix}
0&0\\
0&-\rho(t)
\end{bmatrix}+\begin{bmatrix}
0&0\\
0&-\rho(t)
\end{bmatrix}Q^n(t)\right)\begin{bmatrix}
-I_d\\
\gamma
\end{bmatrix}\\
&\quad-F^n(t)\Lambda^{-1} F^n(t)+\Sigma(t)+2\gamma\rho\Bigg\}\,dt-\begin{bmatrix}
-I_d&\gamma
\end{bmatrix}M^n(t)\begin{bmatrix}
-I_d\\
\gamma
\end{bmatrix}\,dW(t),
\\
F^n(T)=&nI_d-\gamma.
\end{aligned}
\end{equation}
Our goal is to bound the processes $F^n$ by the solutions to deterministic RDEs. To this end, we first prove that the process
\begin{equation*}
\begin{bmatrix}
-I_d&\gamma
\end{bmatrix}\left(Q^n\begin{bmatrix}
0&0\\
0&-\rho
\end{bmatrix}+\begin{bmatrix}
0&0\\
0&-\rho
\end{bmatrix}Q^n\right)\begin{bmatrix}
-I_d\\
\gamma
\end{bmatrix}
\end{equation*}
can be bounded from below and above by $-2F^n$ and $2F^n,$ respectively.
\begin{lemma}\label{keyestimate}
Let $\beta, n_0$ be as in \eqref{n_0} and put
\begin{equation}\label{T_0}
T_0:=\max_{i} \left\{T-\frac{\lambda_{\min}}{\gamma_i(\frac{1}{2}+\beta)}\frac{n_0-(\beta+1) \gamma_i}{n_0-\frac{\gamma_i}{2}} \right\}\vee 0.
\end{equation}
For our choice of $n_0$, we have $T_0< T.$  Then, for any $n\geq n_0$, 
\begin{equation}\label{keyinequality}
\begin{aligned}
-2F^n \leq \begin{bmatrix}
-I_d&\gamma
\end{bmatrix}\left(Q^n\begin{bmatrix}
0&0\\
0&-\rho
\end{bmatrix}+\begin{bmatrix}
0&0\\
0&-\rho
\end{bmatrix}Q^n\right)\begin{bmatrix}
-I_d\\
\gamma
\end{bmatrix}\leq 2F^n,\quad t\in[T_0,T].
\end{aligned}
\end{equation}
\end{lemma}
\begin{proof}
Using the matrix decomposition introduced prior to Section \ref{A Priori Estimates},
we need to prove that
\begin{equation*}
\begin{aligned}
-2F^n \leq -\gamma \rho (B^n)^\mathrm T-B^n\rho\gamma+\gamma C^n\gamma\rho+\rho\gamma C^n\gamma\leq 2F^n,\quad t\in[T_0,T].
\end{aligned}
\end{equation*}
Since $Q^n$ is nonnegative definite, 
\begin{equation*}
\begin{aligned}
&A^n-(2I_d-\rho)\gamma (B^n)^\mathrm T-B^n\gamma(2I_d-\rho)+(2I_d-\rho)\gamma C^n\gamma(2I_d-\rho)\\
=&\begin{bmatrix}
-I_d&(2I_d-\rho)
\end{bmatrix}Q^n(t)\begin{bmatrix}
-I_d\\
(2I_d-\rho)
\end{bmatrix}\geq 0.
\end{aligned}
\end{equation*}
In view of (\ref{def-F}) it follows that,
\begin{equation}\label{computation}
\begin{aligned}
&2F^n+\gamma \rho (B^n)^\mathrm T+B^n\rho\gamma-\gamma C^n\gamma\rho-\rho\gamma C^n\gamma\\
=&[A^n-(2I_d-\rho)\gamma (B^n)^\mathrm T-B^n\gamma(2I_d-\rho)+(2I_d-\rho)\gamma C^n\gamma(2I_d-\rho)]\\
&+A^n-(2I_d-\rho)\gamma C^n\gamma(2I_d-\rho)+2\gamma C^n\gamma-\gamma C^n\gamma\rho-\rho\gamma C^n\gamma\\
\geq& A^n+(I_d+\rho)\gamma C^n\gamma(I_d+\rho)-2\rho\gamma C^n\gamma\rho-3\gamma C^n\gamma\\
\geq& A^n-2\rho\gamma C^n\gamma\rho-3\gamma C^n\gamma.
\end{aligned}
\end{equation}
For $n>\gamma_{\max}$, 
$$\underline A^n_i=\frac{\gamma_i}{e^{\lambda_{\min}^{-1}\gamma_i (T-t)}(1+\frac{\gamma_i}{n-\gamma_i})-1}\geq \frac{\lambda_{\min}}{T-t+\frac{\lambda_{\min}}{n-\gamma_i/2}}-\frac{\gamma_i}{2}.$$
Set
\begin{equation*}
 f(t,n)=\frac{\lambda_{\min}}{T-t+\frac{\lambda_{\min}}{n-\gamma_i/2}}-\frac{\gamma_i}{2}-\beta \gamma_i.
\end{equation*}
It is easy to check that
$$ f(T_0,n_0)\geq \frac{\lambda_{\min}}{\frac{\lambda_{\min}}{\gamma_i(\frac{1}{2}+\beta)}\frac{n_0-(\beta+1) \gamma_i}{n_0-\frac{\gamma_i}{2}} +\frac{\lambda_{\min}}{n_0-\gamma_i/2}}-\frac{\gamma_i}{2}-\beta \gamma_i=0.$$
Since $f$ is increasing in $t$ and $n,$ we have $f(t,n)\geq 0$ for $t\in[T_0,T], n\geq n_0, i=1,\cdots,d.$ Moreover, by Lemma \ref{rough estimate},
$$\gamma C^n\gamma\leq \gamma, \quad \textit{and}\quad A^n\geq diag(\underline A^n_i).$$
Therefore, 
\begin{equation}\label{positive}
A^n-2\rho\gamma C^n\gamma\rho-3\gamma C^n\gamma\geq diag(\underline A^n_i)-\beta \gamma\geq 0.
\end{equation}
This yields the right inequality in \eqref{keyinequality}. For the left inequality, notice that similarly to \eqref{computation}, 
\begin{equation*}
\begin{aligned}
&2F^n-\gamma \rho (B^n)^\mathrm T-B^n\rho\gamma+\gamma C^n\gamma\rho+\rho\gamma C^n\gamma\\
\geq& A^n+(I_d-\rho)\gamma C^n\gamma(I_d-\rho)-2\rho\gamma C^n\gamma\rho-3\gamma C^n\gamma.
\end{aligned}
\end{equation*}
Hence, the left inequality also follows from \eqref{positive}.
\end{proof}

From \cite[Section 2.2.2]{Kratz2011}, we have the following lemma.
\begin{lemma}\label{doubleK}
Let $n>n_1,$ where
\begin{equation*}
n_1:=\max\{\lambda_{\min}(\sqrt{1+\alpha}+1)+\gamma_{\min},\gamma_{\max}\}
\end{equation*} and let $T_0$ be as in equation \eqref{T_0}.
Then the initial value problems
\begin{equation}\label{K^max}
-dK(t)=-\{K(t)^2-2 K(t)-\alpha I_d\}\,dt,\quad K(T)=\frac{n-\gamma_{\min}}{\lambda_{\min}}I_d
\end{equation}
and
\begin{equation}\label{K^min}
-dK(t)=-\{K(t)^2+2 K(t)\}\,dt,\quad K(T)=\frac{n-\gamma_{\max}}{\lambda_{\max}}I_d
\end{equation}
with
\begin{equation}\label{alpha}
\quad \alpha=\frac{||\Sigma||_{L^{\infty}}+2\gamma_{\max}||\rho||_{L^{\infty}}}{\lambda_{\min}}
 \end{equation}
possess unique solutions $K^n_{\max}$ respectively $K^n_{\min}$ on $[T_0,T]$. They are given by
 \begin{equation*}
\begin{aligned}
&K^n_{\max}(t)=p^n(t)I_d,\\
&K^n_{\min}(t)=q^n(t)I_d,
\end{aligned}
\end{equation*}
where
 \begin{equation}\label{keybounds}
\begin{aligned}
&p^n(t)=\sqrt{1+\alpha } \coth (\sqrt{1+\alpha}(T-t)+\kappa^n_1)+1,\\
&q^n(t)=\coth(T-t+\kappa^n_2)-1,
\end{aligned}
\end{equation}
with
\begin{equation*}
\begin{aligned}
&\kappa^n_1=\textnormal{arcoth}(\frac{\frac{n-\gamma_{\min}}{\lambda_{\min}}-1}{\sqrt{1+\alpha}}),\\
&\kappa^n_2=\textnormal{arcoth}(\frac{n-\gamma_{\max}}{\lambda_{\max}}+1).
\end{aligned}
\end{equation*}
\end{lemma}
The matrices $K^n_{\max}, K^n_{\min}$ in Lemma \ref{doubleK} turn out to be the desired bounds for $\sqrt{\Lambda^{-1}}F^n\sqrt{\Lambda^{-1}}$ near the terminal time. 
\begin{proposition}\label{simple estimate}
Let $n_0$ be as in \eqref{n_0} and $T_0$ be as in \eqref{T_0}. 
Then for $n>n_0$, 
\begin{equation}\label{weightedbounds}
q^n(t) I_d\leq \sqrt{\Lambda^{-1}}F^n\sqrt{\Lambda^{-1}}\leq  p^n(t) I_d,\quad t\in[T_0,T].
\end{equation}
\end{proposition}
\begin{proof}
Let $$\hat F^n=\sqrt{\Lambda^{-1}} F^n\sqrt{\Lambda^{-1}}.$$ 
Multiplying $\sqrt{\Lambda^{-1}} $ both on the left and right of \eqref{BSRDE_Fn}, we see that $\hat F^n$ solves
\begin{equation*} 
\begin{aligned}
-d\hat F^n(t)=&\Bigg\{\sqrt{\Lambda^{-1}}\begin{bmatrix}
-I_d&\gamma
\end{bmatrix}\left(Q^n(t)\begin{bmatrix}
0&0\\
0&-\rho(t)
\end{bmatrix}+\begin{bmatrix}
0&0\\
0&-\rho(t)
\end{bmatrix}Q^n(t)\right)\begin{bmatrix}
-I_d\\
\gamma
\end{bmatrix}\sqrt{\Lambda^{-1}}\\
&\quad-\hat F^n(t)\cdot\hat F^n(t)+\sqrt{\Lambda^{-1}}(\Sigma(t)+2\gamma\rho)\sqrt{\Lambda^{-1}}\Bigg\}\,dt\\
&-\sqrt{\Lambda^{-1}}\begin{bmatrix}
-I_d&\gamma
\end{bmatrix}M^n(t)\begin{bmatrix}
-I_d\\
\gamma
\end{bmatrix}\sqrt{\Lambda^{-1}}\,dW(t),
\\
\hat F^n(T)=&\sqrt{\Lambda^{-1}}(nI_d-\gamma)\sqrt{\Lambda^{-1}}.
\end{aligned}
\end{equation*}
From Lemma \ref{keyestimate}, we know that on $[T_0,T],$ 
\begin{equation*}
\begin{aligned}
-2\hat F^n \leq \sqrt{\Lambda^{-1}}\begin{bmatrix}
-I_d&\gamma
\end{bmatrix}\left(Q^n(t)\begin{bmatrix}
0&0\\
0&-\rho(t)
\end{bmatrix}+\begin{bmatrix}
0&0\\
0&-\rho(t)
\end{bmatrix}Q^n(t)\right)\begin{bmatrix}
-I_d\\
\gamma
\end{bmatrix}\sqrt{\Lambda^{-1}}\leq 2\hat F^n.
\end{aligned}
\end{equation*}
In terms of $\alpha$ given in \eqref{alpha}, 
\begin{equation*}
\begin{aligned}
&0\leq\sqrt{\Lambda^{-1}}(\Sigma(t)+2\gamma\rho)\sqrt{\Lambda^{-1}}\leq \alpha I_d,\\
&\frac{n-\gamma_{\max}}{\lambda_{\max}} I_d\leq\sqrt{\Lambda^{-1}}(nI_d-\gamma)\sqrt{\Lambda^{-1}}\leq \frac{n-\gamma_{\min}}{\lambda_{\min}} I_d.
\end{aligned}
\end{equation*}
Applying the Comparison Principle [Theorem \ref{comparison}], we obtain 
$$K^n_{\min}(t)\leq\hat F^n(t)\leq K^n_{\max}(t),\quad t\in[T_0,T]$$
 where $K^n_{\max}, K^n_{\min}$ are the solutions to equations \eqref{K^max}, \eqref{K^min}. Hence the assertion follows from the fact that $K^n_{\max}=p^n I_d$, $K^n_{\min}(t)=q^nI_d.$
\end{proof}

The preceding proposition established upper and lower bounds for the processes $\sqrt{\Lambda^{-1}}F^n\sqrt{\Lambda^{-1}}$ in terms of the functions $q^n$ and $p^n$ on $[T_0,T]$. For analytical convenience we extend these functions and the bounds to the whole interval $[0,T]$ by putting 
\begin{equation}\label{extension}
\begin{aligned}
q^n(t)& \equiv \lambda^{-1}_{\max}\min_i\{\underline F^{n_0}_i(0)\}\\
 p^n(t)&\equiv \lambda^{-1}_{\min}\max_i\{\overline F_i(T_0)\}
\end{aligned}
\end{equation}
for $t\in[0,T_0)$ and $n>n_0.$ Then,  $$q^n(t)I_d\leq \sqrt{\Lambda^{-1}}F^n(t)\sqrt{\Lambda^{-1}}\leq p^n(t) I_d, \quad t \in [0,T].$$

\subsection{Solving the unconstrained problems}\label{penalization}
In this section, we are going to solve the unconstrained optimisation problem. For any initial state $(t,x,y)\in[0,T)\times\mathbb R^d\times \mathbb R^d$, the dynamics of the state process $(X^{n,*},Y^{n,*})$ under the candidate strategy $\xi^{n,*}$ is given by:
\begin{equation} \label{optimal-state-dynamics}
\left\{
\begin{aligned}
dX^{n,*}(s)&=\{-\Lambda^{-1}D^n(s)X^{n,*}(s)+\Lambda^{-1}E^n(s)Y^{n,*}(s)\}\,ds\\
dY^{n,*}(s)&=\{-(\rho(s)+\gamma \Lambda^{-1}E^n(s))Y^{n,*}(s)+\gamma \Lambda^{-1}D^n(s)X^{n,*}(s)\}\,ds.
\end{aligned}
\right.
\end{equation}
In particular,
$dY^{n,*}(s)=-\gamma \,dX^{n,*}(s)-\rho(s)Y^{n,*}(s)\,ds,$
and hence 
\begin{equation}\label{stateY} 
Y^{n,*}(s)=-\gamma X^{n,*}(s)+e^{-\int_t^s \rho(r) dr}(y+\gamma x)+\int_t^s e^{-\int_u^s \rho(r) dr}\gamma \rho(u) X^{n,*}(u)\,du,
\end{equation}
where $e^{-\int_t^s \rho(r) dr}=\textup{diag}(e^{-\int_t^s \rho_i(r) dr}).$ Thus,
\begin{equation}\label{stateX} 
\begin{aligned}
dX^{n,*}(s)=&\left\{ -\Lambda^{-1}(D^n(s)+E^n(s)\gamma )X^{n,*}(s)+\Lambda^{-1}E^n(s)e^{-\int_t^s \rho(r) dr}(y+\gamma x) \right.\\
& \left. +\Lambda^{-1}E^n(s)\int_t^s e^{-\int_u^s \rho(r) dr}\gamma \rho(u) X^{n,*}(u)\,du \right\} ds. 
\end{aligned}
\end{equation}
In order to solve this  linear ordinary differential equation, we introduce the fundamental matrix $\Phi^n(t,s).$ It is given by the unique solution of the ODE system
\begin{equation} \label{Phi}
\left\{
\begin{aligned}
d\Phi^n(t,s)&=-\Lambda^{-1}(D^n(s)+E^n(s)\gamma )\Phi^n(t,s)\,ds,\\
\Phi^n(t,t)&=I_d.
\end{aligned}
\right.
\end{equation}
The inverse $(\Phi^n)^{-1}$ exists and satisfies 
\begin{equation*} 
\left\{
\begin{aligned}
d\Phi^n(t,s)^{-1}&=\Phi^n(t,s)^{-1}\Lambda^{-1}(D^n(s)+E^n(s)\gamma )\,ds,\\
\Phi^n(t,t)^{-1}&=I_d.
\end{aligned}
\right.
\end{equation*}

The following lemma establishes norm bounds on the fundamental solution and its inverse. 

\begin{lemma}\label{fundamental}
Let us fix $t\in [0,T)$. For all $t\leq s\leq T,$
\begin{equation} \label{fundamental-estimate}
\begin{aligned}
&|\Phi^{n}(t,s)|^2\leq  {d \frac{\lambda_{\max}}{\lambda_{\min}}}\exp\left(-2 \int_t^s q^n(u)\,du\right),\\
&|\Phi^{n}(t,s)^{-1}|^2\leq {d \frac{\lambda_{\max}}{\lambda_{\min}}}\exp\left(2 \int_t^s p^n(u)\,du\right).
\end{aligned}\
\end{equation}
In particular, $|\Phi^{n}(t,\cdot)|$ is uniformly bounded on $[t,T]$, due to Lemma \ref{appendix-estimate}.
\end{lemma}
\begin{proof}
Let $\Phi^n(t,s)=\begin{bmatrix}
\phi_1^n(t,s)&\phi_2^n(t,s)&\cdots&\phi_d^n(t,s)
\end{bmatrix}.$
For $i=1,...,d,$ we obtain by Proposition \ref{simple estimate} and (\ref{extension}) that
\begin{equation*} 
\begin{aligned}
d\phi_i^{n}(t,s)^\mathrm T\Lambda\phi_i^{n}(t,s)&=-2\phi_i^{n}(t,s)^\mathrm T(D^n(s)+E^n(s)\gamma )\phi_i^n(t,s)\,ds\\
&= -2\phi_i^{n}(t,s)^\mathrm T\sqrt{\Lambda}\hat F^n(s)\sqrt{\Lambda}\phi_i^n(t,s)\,ds\\
&\leq -2 q^n(s)\phi_i^{n}(t,s)^\mathrm T\Lambda\phi_i^{n}(t,s).
\end{aligned}
\end{equation*}
Since $q^n(s)$ is discontinuous at $T_0,$  if $t<T_0<s$ we divide the interval $[t,s)$ into two subintervals $[t,T_0), [T_0,s)$. On each subinterval the assumptions of Gronwall's inequality are satisfied. Hence, 
\begin{equation} \label{discontinuity}
\begin{aligned}
&\phi_i^{n}(t,T_0)^\mathrm T\Lambda\phi_i^{n}(t,T_0)\leq \phi_i^{n}(t,t)^\mathrm T\Lambda\phi_i^{n}(t,t)\exp\left(-2\int_t^{T_0} q^n(u)\,du\right), \\
&\phi_i^{n}(t,s)^\mathrm T\Lambda\phi_i^{n}(t,s)\leq \phi_i^{n}(t,T_0)^\mathrm T\Lambda\phi_i^{n}(t,T_0)\exp\left(-2\int_{T_0}^s q^n(u)\,du\right).
\end{aligned}\
\end{equation}
Hence, for all $t \leq s \leq T$,
$$\phi_i^{n}(t,s)^\mathrm T\Lambda\phi_i^{n}(t,s)\leq \phi_i^{n}(t,t)^\mathrm T\Lambda\phi_i^{n}(t,t)\exp\left(-2\int_t^s q^n(u)\,du\right).$$
Since $\Lambda$ is positive definite, and because $\phi^n_i(t,t)$ is the $i$th unit vector this yields 
\begin{equation*} 
\begin{aligned}
|\phi_i^{n}(t,s)|^2 = \phi_i^{n}(t,s)^\mathrm T I_d \, \phi_i^{n}(t,s)  &\leq\frac{1}{\lambda_{\min}}\phi_i^{n}(t,s)^\mathrm T\Lambda\phi_i^{n}(t,s)
\leq \frac{\lambda_{\max}}{\lambda_{\min}}\exp\left(-2\int_t^s q^n(u)\,du\right).
\end{aligned}
\end{equation*}
Hence, 
\begin{equation*} 
\begin{aligned}
|\Phi^{n}(t,s)|^2=\sum\limits_{1\leq i\leq d}|\phi_i^{n}(t,s)|^2&\leq d \frac{\lambda_{\max}}{\lambda_{\min}}\exp\left(-2\int_t^s q^n(u)\,du\right).
\end{aligned}
\end{equation*}
This proves the desired bound for the fundamental solution.  
Since $|(\Phi^n(t,s))^{-1}|=|[(\Phi^n(t,s))^{-1}]^\mathrm T|,$ we may consider 
the differential equation 
\begin{equation*} 
\left\{
\begin{aligned}
d[(\Phi^n(t,s))^{-1}]^\mathrm T&=(D^n(s)+E^n(s)\gamma )\Lambda^{-1}[(\Phi^n(t,s))^{-1}]^\mathrm T\,ds,\\
[(\Phi^n(t,s))^{-1}]^\mathrm T&=I_d,
\end{aligned}
\right.
\end{equation*}
in order to establish the desired bound for the inverse. This system is similar to \eqref{Phi}. The desired bounds thus follow from similar arguments as before. 
%
\end{proof}

The following bounds on the state process $(X^{n,*},Y^{n,*})$ are key to our subsequent analysis.
\begin{proposition}\label{prop-X}
	Let $n>n_0$ for $n_0$  as in \eqref{n_0}. Then there exists a constant $C >0$ that is independent of $n,$ such that for all $s\in [t,T],$
\begin{equation}\label{xestimate}
\begin{aligned}
&|X^{n,*}(s)|\leq C |\Phi^n(t,s)|,\\
&|Y^{n,*}(s)|\leq C.
\end{aligned}
\end{equation}
\end{proposition}
\begin{proof}
Let $\tilde X^{n,*}(s)=\Phi^n(t,s)^{-1}X^{n,*}(s).$ Differentiating this equation and using \eqref{stateX} yields, 
\begin{equation*} 
\begin{aligned}
\tilde X^{n,*}(s)=&x+\int_t^s \Phi^{n}(t,r)^{-1}\Lambda^{-1}E^n(r)\big\{\\
&e^{-\int_t^r \rho(u)\,du}(y+\gamma x)+\int_t^r e^{-\int_u^r \rho(v)\,dv}\gamma \rho(u)\Phi^n(t,u)\tilde X^{n,*}(u)\,du\big\}\,dr.
\end{aligned}
\end{equation*}
Since $\rho\geq 0,$ 
%
the interated integral version of Gronwall's inequality yields
\begin{equation} \label{Xestimate}
\begin{aligned}
|\tilde X_s^{n,*}|\leq&\left[|x|+\int_t^s |\Phi^{n}(t,r)^{-1}\Lambda^{-1}E^n(r)|\cdot|y+\gamma x|\,dr\right]\\
&\cdot \exp\left(\int_t^s |\Phi^{n}(t,r)^{-1}\Lambda^{-1}E^n(r)|\int_t^r |\gamma \rho(u)\Phi^n(t,u)|\,du\,dr\right).
\end{aligned}
\end{equation}
The processes $\Phi^n$ and $\rho$ are bounded. Moreover, we prove below that there exists a constant  $C > 0$, which is independent of $n$ such that for  $s\in [t,T],$
\begin{equation} \label{boundE}
\begin{aligned}
|E^n(s)|\leq 
C(T-s).
\end{aligned}
\end{equation}
Then the desired bounds follow from Lemma \ref{fundamental} and Lemma \ref{appendix-estimate} as 
\begin{equation*}
\begin{aligned}
\int_t^T |\Phi^{n}(t,r)^{-1}\Lambda^{-1}E^n(r)|\,dr&\leq \int_t^T |\Phi^{n}(t,r)^{-1}|\cdot|\Lambda^{-1}|\cdot|E^n(r)|\,dr\\
& \leq\int_t^T\sqrt{\frac{d\lambda_{\max}}{\lambda_{\min}}}\exp\left(\int_t^s p^n(u)\,du\right)\cdot|\Lambda^{-1}|\cdot L(T-r)\,dr\\
&\leq |\Lambda^{-1}|\sqrt{\frac{d\lambda_{\max}}{\lambda_{\min}}}[\int_t^{T_0} \mathbb I_{t\in[0,T_0)}L\cdot C (T-r)\,dr+\int_{T_0}^T L\cdot C\,dr]\\
&<\infty.
\end{aligned}
\end{equation*}

In order to establish the bound (\ref{boundE}) 
we multiply $\begin{bmatrix}
-I_d&\gamma
\end{bmatrix}$ on the left and $\begin{bmatrix}
0\\
I_d
\end{bmatrix}$ on the right in \eqref{BSRDE'} and use the decomposition of the matrix $Q$ introduced prior to Section 3.1. Thus, 
\begin{equation*}
\left\{
\begin{aligned}
-dE^n(s)=&\{-(D^n(s)+E^n(s)\gamma )\Lambda^{-1}E^n(s)-E^n(s)\rho(s)-\rho(s)\gamma C^n(s)+2\rho(s)\}\,ds\\
&\quad-M^n_{E}(s)\,dW(s),\\
E^n(T)=&0,
\end{aligned}
\right.
\end{equation*}
where $M^n_{E}:=\begin{bmatrix}
-I_d&\gamma
\end{bmatrix}M^n\begin{bmatrix}
0\\
I_d
\end{bmatrix}\in L_{\mathcal F}^2(0,T;\mathbb R^{d\times d})$. Recalling the definition of $\Phi^n$ in \eqref{Phi}, 
\begin{equation*}
\begin{aligned}
&-d\left[\Phi^{n}(t,s)^\mathrm TE^n(s)e^{-\int_t^s\rho(u)\,du}\right]\\
=&-\Phi^{n}(t,s)^\mathrm T\left[dE^n(s)+(D^n(s)+E^n(s)\gamma )\Lambda^{-1}E^n(s)+E^n(s)\rho(s)\right]e^{-\int_t^s\rho(u)\,du}\\
=& \Phi^{n}(t,s)^\mathrm T(-\rho(s)\gamma C^n(s)+2\rho(s))e^{-\int_t^s\rho(u)\,du}\,ds-\Phi^{n}(t,s)^\mathrm T M^n_{E}(s)e^{-\int_t^s\rho(u)\,du}\,dW(s).
\end{aligned}
\end{equation*}

The uniform boundedness of $\Phi^n(t,\cdot)$ together with $\rho\geq 0$ and $M^n_E \in L_{\mathcal F}^2(0,T;\mathbb R^{d\times d})$ yields, 
\begin{equation*}
\begin{aligned}
E^n(s)=\mathbb E^{\mathcal F_s}\Bigg\{&[(\Phi^{n}(t,s))^\mathrm T]^{-1}\Phi^{n}(t,T)^\mathrm TE^n(T)e^{-\int_s^T\rho(u)\,du}\\
&+[(\Phi^{n}(t,s))^\mathrm T]^{-1}\int_s^T \Phi^{n}(t,r)^\mathrm T(-\rho(r)\gamma C^n(r)+2\rho(r))e^{-\int_s^r\rho(u)\,du}\,dr\Bigg\}.
\end{aligned}
\end{equation*}

Let $\Psi^n(s,r)=\Phi^{n}(t,r)(\Phi^{n}(t,s))^{-1}.$ Then $\Psi^n(s,\cdot)$ satisfies
\begin{equation*} 
\left\{
\begin{aligned}
d\Psi^n(s,r)&=-\Lambda^{-1}(D^n(r)+E^n(r)\gamma )\Psi^n(s,r)\,dr\\
\Psi^n(s,s)&=I_d.
\end{aligned}
\right.
\end{equation*}
In view of Lemma \ref{fundamental}, $\Psi^n(s,\cdot)$ is uniformly bounded on $[s,T]$. Hence (\ref{boundE}) follows from $E^n(T)=0$ along with the boundedness of $\rho$ and the uniform boundedness of the matrices $C^n$; cf.~Lemma \ref{rough estimate}. 
\end{proof}

\begin{proposition} \label{prop-admissible}
	Let $n>n_0$ for $n_0$ as in \eqref{n_0}. Then the feedback control $\xi^{n,*}$ given in \eqref{feedback-form} is admissible. Moreover, $\xi^{n,*}$ is uniformly bounded.
\end{proposition}
\begin{proof}
From \eqref{stateX}, we know that for $s\in [t,T],$
\begin{equation}\label{statexi} 
\begin{aligned}
\xi^{n,*}(s)=&\Lambda^{-1}(D^n(s)+E^n(s)\gamma )X^{n,*}(s)-\Lambda^{-1}E^n(s)e^{-\int_t^s \rho(r) dr}(y+\gamma x)\\
&-\Lambda^{-1}E^n(s)\int_t^s e^{-\int_u^s \rho(r) dr}\gamma \rho(u) X^{n,*}(u)\,du.
\end{aligned}
\end{equation}
Since the portfolio processes are uniformly bounded, due to Proposition \ref{prop-X} and the processes $E^n$ are uniformly bounded, due to (\ref{boundE}) it is enough to establish an $L^\infty$-bound for the first term on the right side in \eqref{statexi}. 

For this, we recall \cite[Theorem 8.4.9]{Bernstein2005} that $|A|_{2,2}\leq |B|_{2,2}$ for any two symmetric matrices $0\leq A\leq B$. Thus, by Proposition \ref{simple estimate}, Proposition \ref{prop-X} and Lemma \ref{appendix-estimate}, 
\begin{equation*}
\begin{aligned}
\sup_{n>n_0}|\Lambda^{-1}(D^n(s)+E^n(s)\gamma )X^{n,*}(s)| 
&\leq \sup_{n>n_0} |\Lambda^{-1}(D^n(s)+E^n(s)\gamma )|_{2,2}|X^{n,*}(s)|\\
&=\sup_{n>n_0} |\sqrt{\Lambda^{-1}}F^n\sqrt{\Lambda^{-1}}|_{2,2}|X^{n,*}(s)|\\
&\leq \sup_{n>n_0}p^n(s)\cdot C|\Phi^n(t,s)|\\
&\leq \sup_{n>n_0}p^n(s)\cdot C\sqrt{\frac{d\lambda_{\max}}{\lambda_{\min}}}\exp\left(-\int_t^s q^n(u)\,du\right)\\
&<\infty.
\end{aligned}
\end{equation*}
This establishes a uniform $L^\infty$-bound on the optimal trading strategies. 
%
\end{proof}
We are now ready to carry out the verification argument for the unconstrained problem \eqref{control-modified}.
\begin{proof}[Proof of Theorem \ref{OPT-modified}]
Let us fix an initial state $(t,x,y)\in [0,T)\times \mathbb R^d\times\mathbb R^d$ and admissible control $\xi^n\in L_{\mathcal F}^2(t,T;\mathbb R^{d})$. For $R\in \mathbb R_+$ we define the stopping time
$$\tau_R:=\inf\{t\leq s\leq T:|X^n(s)|\vee|Y^n(s)|\geq R\}.$$
Since $(V^n,N^n)$ solves the HJB equation, standard arguments show that for any $\xi^n\in L^2_{\mathcal F}([t,T];\mathbb R^d)$,
%
%
\begin{equation}\label{HJBequation2} 
{\small
\begin{split}
V^n(t,x,y)\leq \mathbb E\bigg\{&nX^n(T\wedge\tau_R)^\mathrm{T}X^n(T\wedge\tau_R)+2Y^n(T\wedge\tau_R)^\mathrm{T}X^n(T\wedge\tau_R)\\
&+\int_t^{T\wedge\tau_R}\big[\frac{1}{2}\xi^n(s)^\mathrm{T}\Lambda\xi^n(s)+Y^n(s)^\mathrm{T}\xi^n(s)+\frac{1}{2}X^n(s)^\mathrm{T}\Sigma(s) X^n(s)\big]\,ds|\mathcal F_t\bigg\}\\
\end{split}}
\end{equation}
where the inequality is an equality if $\xi^n = \xi^{n,*}$. Since $P^n\in L^\infty_{\mathcal F}, X^n, Y^n \in L^2_{\mathcal F}$ and $\Sigma \in\L^\infty_{\mathcal F}$, 
it follows from H\"older's inequality that
$$\begin{bmatrix}
 (X^n)^\mathrm{T}&(Y^n)^\mathrm{T}
\end{bmatrix}P^n \begin{bmatrix}
\ X^n\ \\
Y^n
\end{bmatrix}\in L^1_{\mathcal F}(\Omega;C([t,T];\mathbb R)),$$
and
$$(\xi^n)^\mathrm{T}\Lambda\xi^n, (Y^n)^\mathrm{T}\xi^n, (X^n)^\mathrm{T}\Sigma X^n\in L^1_{\mathcal F}([t,T];\mathbb R).$$
Thus, the dominated convergence theorem applies when letting $R\rightarrow\infty$ in \eqref{HJBequation2}. This yields
\begin{equation}\label{HJBequation3} 
\begin{aligned}
V^n(t,x,y)\leq \mathbb E\bigg\{&nX^n(T)^\mathrm{T}X^n(T)+2Y^n(T)^\mathrm{T}X^n(T)\\
&+\int_t^{T}\big[\frac{1}{2}\xi^n(s)^\mathrm{T}\Lambda\xi^n(s)+Y^n(s)^\mathrm{T}\xi^n(s)+\frac{1}{2}X^n(s)^\mathrm{T}\Sigma(s) X^n(s)\big]\,ds|\mathcal F_t\bigg\}\\
\end{aligned}
\end{equation}
with equality if $\xi^n=\xi^{n,*}.$  
\end{proof}

\subsection{Solving the optimal liquidation problem}\label{original problem}

\subsubsection{The candidate value function}\label{candidate}
In this section, we prove that the limit $$Q(t):=\lim\limits_{n\rightarrow \infty}Q^n(t)$$ exists for $t\in [0,T)$. In particular, the candidate value function for \eqref{control-problem},
$$V(t,x,y):=\begin{bmatrix}
 x^\mathrm{T}&y^\mathrm{T}
\end{bmatrix}\left(Q(t)-\begin{bmatrix}
0&0\\
0&\gamma^{-1}
\end{bmatrix}\right)\begin{bmatrix}
\ x\ \\
y 
\end{bmatrix},$$
is well-defined.
\begin{proof}[Proof of Theorem \ref{existence}]
For given $t\in[0,T),$ the sequence $\{Q^n(t)\}$ is nondecreasing. Moreover, the a priori estimates \eqref{generalestimate} imply that $$|Q^n|\leq \sqrt{|A^n_{\max}|^2+|B^n_{\max}|^2+|C^n_{\max}|^2} \leq C$$ for some constant $C >0$ uniformly on $[0,t].$ In particular, the sequence $\{Q^n(\cdot)\}$ converges pointwise and in $L^2$ to some limiting process $Q(\cdot)$ on $[0,t]$. Using the continuity of $Q^n$, 
$$\liminf\limits_{t\rightarrow T}|Q(t)|\geq \liminf\limits_{t\rightarrow T}|Q(t)|_{2,2}\geq \liminf\limits_{t\rightarrow T}|Q^n(t)|_{2,2}=|Q^n(T)|_{2,2}>n.$$
This shows that $$\liminf\limits_{t\rightarrow T}|Q(t)|= +\infty.$$

We are now going to show that $Q$ is one part of the solution to the matrix differential equation \eqref{limitequation} on $[0,T).$ To this end, let $n>m,$ and let $(Q^n,M^n), (Q^m,M^m)$ be the solutions of \eqref{BSRDE'} with terminal values $\begin{bmatrix}
nI_d&I_d\\
I_d&\gamma^{-1}
\end{bmatrix}$ and $\begin{bmatrix}
mI_d&I_d\\
I_d&\gamma^{-1}
\end{bmatrix},$ respectively.
 Applying the It\^o formula to $|Q^n-Q^m|^2$ on $[s, t],$ we obtain,
\begin{equation}\label{cauchy}
\begin{aligned}
&|Q^n(s)-Q^m(s)|^2+\int^t_s|M^n(r)-M^m(r)|^2\,dr\\
=&|Q^n(t)-Q^m(t)|^2-2\int^t_s \textup{tr}\Big((Q^n(r)-Q^m(r))(M^n(r)-M^m(r))\Big)\,dW(r)\\
&\quad+2\int^t_s \textup{tr}\Big((Q^n(r)-Q^m(r))[g(r,Q^n(r))-g(r,Q^m(r))]\Big)\,dr,
\end{aligned}
\end{equation}
where
\begin{equation*}
\begin{aligned}
g(r,Q(r)):=&-Q(r)\begin{bmatrix}
-I_d\\
\gamma
\end{bmatrix}\Lambda^{-1}\begin{bmatrix}
-I_d&\gamma
\end{bmatrix}Q(r)+Q(r)\begin{bmatrix}
0&0\\
0&-\rho(r)
\end{bmatrix}+\begin{bmatrix}
0&0\\
0&-\rho(r)
\end{bmatrix}Q(r)\\
&+\begin{bmatrix}
\Sigma(r)&0\\
0&\gamma^{-1}\rho(r)+\rho(r)\gamma^{-1}
\end{bmatrix}
\end{aligned}
\end{equation*}
and
\begin{equation*}
\begin{aligned}
g(r,Q^n(r))-g(r,Q^m(r))=&-(Q^n(r)-Q^m(r))\begin{bmatrix}
-I_d\\
\gamma
\end{bmatrix}\Lambda^{-1}\begin{bmatrix}
-I_d&\gamma
\end{bmatrix}(Q^n(r)-Q^m(r))\\
&+\Big(\begin{bmatrix}
0&0\\
0&-\rho(r)
\end{bmatrix}-Q^m(r)\begin{bmatrix}
-I_d\\
\gamma
\end{bmatrix}\Lambda^{-1}\begin{bmatrix}
-I_d&\gamma
\end{bmatrix}\Big)(Q^n(r)-Q^m(r))\\
&+(Q^n(r)-Q^m(r))\Big(\begin{bmatrix}
0&0\\
0&-\rho(r)
\end{bmatrix}-\begin{bmatrix}
-I_d\\
\gamma
\end{bmatrix}\Lambda^{-1}\begin{bmatrix}
-I_d&\gamma
\end{bmatrix}Q^m(r)\Big).
\end{aligned}
\end{equation*}
Due to the symmetry of $Q^n(r)$ and monotonicity of the sequence $\{Q^n(r)\}$, the square root $\sqrt{Q^n(r)-Q^m(r)}$ exists. Since 
$\Lambda^{-1}$ is positive definite, 
\begin{equation*}
\begin{aligned}
&\textup{tr}\Big((Q^n(r)-Q^m(r))[-(Q^n(r)-Q^m(r))\begin{bmatrix}
-I_d\\
\gamma
\end{bmatrix}\Lambda^{-1}\begin{bmatrix}
-I_d&\gamma
\end{bmatrix}(Q^n(r)-Q^m(r))]\Big)\\
&=-\textup{tr}\Big((Q^n(r)-Q^m(r))^{\frac{3}{2}}\begin{bmatrix}
-I_d\\
\gamma
\end{bmatrix}\Lambda^{-1}\begin{bmatrix}
-I_d&\gamma
\end{bmatrix}(Q^n(r)-Q^m(r))^{\frac{3}{2}}\Big)\leq 0.
\end{aligned}
\end{equation*}
Since the sequence $\{Q^n\}$ is uniformly bounded on $[0,t]$ and $\rho, \Sigma \in L^\infty_{\mathcal F}(0,T;\mathcal S^d_+),$
$$\sup_{0\leq r\leq t}\left|\Big(\begin{bmatrix}
0&0\\
0&-\rho(r)
\end{bmatrix}-\begin{bmatrix}
-I_d\\
\gamma
\end{bmatrix}\Lambda^{-1}\begin{bmatrix}
-I_d&\gamma
\end{bmatrix}Q^m(r)\Big)\right|\leq C,$$
for some constant $C > 0$ that is independent of $n, m.$ Using $\textup{tr}(AB)\leq |A|\cdot|B|,$
\begin{equation}\label{generator}
\textup{tr}\Big((Q^n(r)-Q^m(r))[g(r,Q^n(r))-g(r,Q^m(r))]\Big)\leq  C|Q^n(r)-Q^m(r)|^2.
\end{equation}
Moreover, $M^n, M^m \in L_{\mathcal F}^2(0,T;\mathcal S^{2d})$ yields,   
$$\mathbb E\Big[\int^t_s \textup{tr}\Big((Q^n(r)-Q^m(r))(M^n(r)-M^m(r))\Big)\,dW(r)\Big]=0.$$
Hence, 
\begin{equation}\label{Mn}
\begin{aligned}
\mathbb E\int^t_s|M^n(r)-M^m(r)|^2\,dr
\leq \mathbb E\Big[|Q^n(t)-Q^m(t)|^2+C \int^t_s |Q^n(r)-Q^m(r)|^2\,dr\Big].
\end{aligned}
\end{equation}
Using the Burkholder-Davis-Gundy inequality in \eqref{cauchy} yields a constant $C > 0$ such that
\begin{equation}\label{BDG}
\begin{aligned}
&\mathbb E\sup_{0\leq s\leq t}|Q^n(s)-Q^m(s)|^2\\
\leq&\mathbb E\Big[|Q^n(t)-Q^m(t)|^2+\int^t_0 C |Q^n(r)-Q^m(r)|^2\,dr\Big]\\
&+ C\mathbb E\left[\sqrt{\int_0^t|Q^n(r)-Q^m(r)|^2|M^n(r)-M^m(r)|^2\,dr}\right].
\end{aligned}
\end{equation}
By Young's inequality, 
\begin{equation*}
\begin{aligned}
&C\mathbb E\left[\sqrt{\int_0^t|Q^n(r)-Q^m(r)|^2|M^n(r)-M^m(r)|^2\,dr}\right]\\
\leq&\frac{1}{2}\mathbb E\sup_{0\leq s\leq t}|Q^n(s)-Q^m(s)|^2+C\mathbb E\int_0^t|M^n(r)-M^m(r)|^2\,dr.
\end{aligned}
\end{equation*}
Altogether, we arrive at
\begin{equation*}
\begin{aligned}
\mathbb E\sup_{0\leq s\leq t}|Q^n(s)-Q^m(s)|^2\,ds\leq&C \mathbb E\left[|Q^n(t)-Q^m(t)|^2+\int^t_0 |Q^n(r)-Q^m(r)|^2dr\right].
\end{aligned}
\end{equation*}
The right-hand side converges to zero as $n, m \rightarrow \infty.$ This shows that$$Q\in L_{\mathcal F}^\infty(\Omega;C([0,T^-];\mathcal S^{2d}_+)).$$ Furthermore, \eqref{Mn} implies that $\{M^n\}$ is a Cauchy sequence in $L_{\mathcal F}^2(0,t;\mathcal S^{2d})$ and converges to some $M \in L_{\mathcal F}^2(0,t;\mathcal S^{2d})$, for every $t<T.$ Taking the limit $n\rightarrow \infty$ in \eqref{BSRDE'} implies $(Q, M)$ satisfies the matrix differential equation \eqref{limitequation} on $[0,T).$ Compact convergence  follows by Dini's theorem, due to the monotonicity.
\end{proof}

\subsubsection{Verification}\label{Verification}
Before proving that the strategy $\xi^*$ defined in Theorem \ref{OPT} is admissible, we first analyse the controlled processes $X^*, Y^*$ and show that $\xi^*$ is a liquidation strategy, i.e.~that
$\lim\limits_{s\rightarrow T}X^*(s)=0.$
\begin{proposition}
\begin{itemize}
\item[(i)] Let ${Z^{n,*}}^\mathrm T=({X^{n,*}}^\mathrm T,{Y^{n,*}}^\mathrm T),  {Z^{*}}^\mathrm T=({X^{*}}^\mathrm T,{Y^{*}}^\mathrm T).$ Then
$$Z^{n,*}\stackrel{n\rightarrow \infty}{\longrightarrow}Z^*  \textit{ compactly on } [t,T).$$
\item[(ii)]$$n|X^{n,*}(T)|^2\stackrel{n\rightarrow \infty}{\longrightarrow}0$$
In particular,
\begin{equation} 
\begin{aligned}
&X^{n,*}(T)\stackrel{n\rightarrow \infty}{\longrightarrow}X^*(T)=\lim\limits_{s\rightarrow T}X^*(s)=0.\\
&Y^{n,*}(T)^\mathrm{T}X^{n,*}(T)\stackrel{n\rightarrow \infty}{\longrightarrow} Y^{*}(T)^\mathrm{T}X^{*}(T)=0
\end{aligned}
\end{equation}
\end{itemize}
\end{proposition}
\begin{proof}
(i)\,Let $t\leq T'<T.$ On $[t,T'],  Z^*$ and $Z^{n,*}$ solve the differential equations
\begin{equation*} 
\begin{aligned}
&dZ=
\left(-\begin{bmatrix}
-I_d\\
\gamma
\end{bmatrix}\Lambda^{-1}\begin{bmatrix}
-I_d&\gamma
\end{bmatrix}R+\begin{bmatrix}
0&0\\
0&-\rho
\end{bmatrix}\right)Z 
\end{aligned}
\end{equation*}
for $R=Q$ and $R=Q^n$, respectively. Since on $[0,T']$ the sequences $\{Q^n\}$ and $\{Z^{n,*}\}$ are uniformly bounded  and $\{Q^n\}$ uniformly converges to $Q$, the first assertion follows from the continuous dependence of solutions of systems of ordinary linear differential equations on the right side.

(ii) The convergence of the sequence $\{n|X^{n,*}(T)|^2\}$ to zero follows from Proposition \ref{prop-X} along with Lemma \ref{fundamental} and Lemma \ref{appendix-estimate}. 
From the bound on $\sqrt{\Lambda^{-1}} F^n\sqrt{\Lambda^{-1}}$ in \eqref{weightedbounds} and the definition of $ p^n,  q^n$ in \eqref{keybounds}, \eqref{extension},  we know that
$$\lim\limits_{n\rightarrow \infty}\sqrt{\Lambda^{-1}} F^n\sqrt{\Lambda^{-1}}=\sqrt{\Lambda^{-1}} F\sqrt{\Lambda^{-1}}$$ can be bounded from above and below by $ p(s):=\lim\limits_{n\rightarrow \infty} p^n(s)$ and $q(s):=\lim\limits_{n\rightarrow \infty}  q^n(s),$ respectively. \\
Therefore,  similar argument to the proof of Proposition \ref{prop-X} show that
\begin{equation*}
\begin{aligned}
|X^{*}(s)|\leq  C|\Phi(t,s)|\leq C\sqrt{\frac{d\lambda_{\max}}{\lambda_{\min}}}e^{-\int^s_t  q(u)du}.
\end{aligned}
\end{equation*}
By Lemma \ref{appendix-estimate}, 
$\lim\limits_{s\rightarrow T}e^{\int^s_t- q(u)du}=\lim\limits_{s\rightarrow T}\lim\limits_{n\rightarrow \infty}e^{\int^s_t- q^n(u)du}=0$,
which yields
$$\lim\limits_{s\rightarrow T}X^*(s)=0.$$
Using the uniform boundedness of $|Y^{n,*}(s)|$, similar arguments show that
$$\lim\limits_{s\rightarrow T}Y^{n,*}(s)^\mathrm TX^{n,*}(s)=Y^{*}(T)^\mathrm TX^{*}(T)=0.$$
\end{proof}

From the compact convergence results on $Q^n, X^{n,*},Y^{n,*},$ we know that
$$\xi^{n,*}(\cdot,X^{n,*}(\cdot),Y^{n,*}(\cdot))\stackrel{n\rightarrow \infty}{\longrightarrow}\xi^{*}(\cdot,X^{*}(\cdot),Y^{*}(\cdot)) \textit{ compactly on } [t,T).$$
\begin{proposition}\label{liquidation strategy}
	The feedback control $\xi^{*}$  is an admissible liquidation strategy.
\end{proposition}
\begin{proof}
Similarly to the proof of Proposition \ref{prop-admissible}, for $s\in[t,T],$
\begin{equation*}
\begin{aligned}
\xi^{*}(s)=&\Lambda^{-1}(D(s)+E(s)\gamma )X^{*}(s)-\Lambda^{-1}E(s)e^{-\int_t^s \rho(r) dr}(y+\gamma x)\\
&-\Lambda^{-1}E(s)\int_t^s e^{-\int_u^s \rho(r) dr}\gamma \rho(u) X^{*}(u)\,du,
\end{aligned}
\end{equation*}
and
\begin{equation*}
\begin{aligned}
|\Lambda^{-1}(D(s)+E(s)\gamma )X^{*}(s)|&\leq p(s)|X^{*}(s) |< \infty.
\end{aligned}
\end{equation*}
Therefore,
\begin{equation*}
\begin{aligned}
|\xi^{*}(s)|\leq & |\Lambda^{-1}(D(s)+E(s)\gamma )X^{*}(s)|+|\Lambda^{-1}E(s)e^{-\int_t^s \rho(r) dr}(y+\gamma x)| \\
&+|\Lambda^{-1}E(s)\int_t^s e^{-\int_u^s \rho(r) dr}\gamma \rho(u) X^{*}(u)\,du|<\infty.
\end{aligned}
\end{equation*}
Hence $$\xi^{*}\in L^\infty_{\mathcal F}(\Omega;C([t,T];\mathbb R^d)).$$ 
\end{proof}
We are now ready to verify that the limit of the solution to \eqref{control-modified} is indeed the solution to the original problem \eqref{control-problem}.
\begin{proof}[Proof of Therorem \ref{OPT}]
We fix $(t,x,y)\in [0,T)\times \mathbb R^d\times\mathbb R^d$. Noticing that $\xi^{n,*}, X^{n,*}, Y^{n,*} $ are uniformly bounded on $[t, T]$ and respectively converge to $\xi^{*}, X^{*}, Y^{*}$ as $n\rightarrow +\infty,$ we can apply the dominated convergence theorem to obtain 
\begin{equation} 
\begin{aligned}
&\lim\limits_{n\rightarrow \infty}V^n(t,x,y) \\
&\begin{aligned}
=\lim\limits_{n\rightarrow \infty} \mathbb E\bigg\{
&nX^n(T)^\mathrm{T}X^{n,*}(T)+2Y^{n,*}(T)^\mathrm{T}X^{n,*}(T)+\int_t^T\big[\frac{1}{2}\xi^{n,*}(s)^\mathrm{T}\Lambda\xi^{n,*}(s)\\
&+Y^{n,*}(s)^\mathrm{T}\xi^{n,*}(s)+\frac{1}{2}X^{n,*}(s)^\mathrm{T}\Sigma(s) X^{n,*}(s)\big]\,ds|\mathcal F_t\bigg\}.\end{aligned}\\
&\begin{aligned}=\mathbb E\bigg\{
&\lim\limits_{n\rightarrow \infty}nX^n(T)^\mathrm{T}X^{n,*}(T)+2Y^{n,*}(T)^\mathrm{T}X^{n,*}(T)+\lim\limits_{n\rightarrow \infty}\int_t^T\big[\frac{1}{2}\xi^{n,*}(s)^\mathrm{T}\Lambda\xi^{n,*}(s)\\
&+Y^{n,*}(s)^\mathrm{T}\xi^{n,*}(s)+\frac{1}{2}X^{n,*}(s)^\mathrm{T}\Sigma(s) X^{n,*}(s)\big]\,ds|\mathcal F_t\bigg\}.\end{aligned}\\
&= \mathbb E\bigg\{\int_t^T\lim\limits_{n\rightarrow \infty}\big[\frac{1}{2}\xi^{n,*}(s)^\mathrm{T}\Lambda\xi^{n,*}(s)+Y^{n,*}(s)^\mathrm{T}\xi^{n,*}(s)+\frac{1}{2}X^{n,*}(s)^\mathrm{T}\Sigma(s) X^{n,*}(s)\big]\,ds|\mathcal F_t\bigg\}\\
&=\mathbb E\bigg\{\int_t^T\big[\frac{1}{2}\xi^{*}(s)^\mathrm{T}\Lambda\xi^{*}(s)+Y^{*}(s)^\mathrm{T}\xi^{*}(s)+\frac{1}{2}X^{*}(s)^\mathrm{T}\Sigma(s) X^{*}(s)\big]\,ds|\mathcal F_t\bigg\}\\
&=V(t,x,y)\\
\end{aligned}
\end{equation}
Therefore, $\xi^*$ solves the Optimization Problem \eqref{control-problem} and the value function is given by $V$.
\end{proof}

\section{Numerical analysis} 
In this section we simulate the solution to a deterministic benchmark model with two assets. In order to simplify the exposition, we assume that there is no cross asset price impact and choose
\begin{equation}
\Lambda=\begin{bmatrix}
\lambda_1&0\\
0 &\lambda_2
\end{bmatrix},\qquad\Sigma=\begin{bmatrix}
\sigma_1^2&k\sigma_1\sigma_2\\\
k\sigma_1\sigma_2&\sigma_2^2
\end{bmatrix}.
\end{equation}

If all the cost coefficients are deterministic constants, the stochastic Riccati equation reduces to a multi-dimensional ODE system that can be solved numerically using the MATLAB package \texttt{bvpsuite}~\cite{Kitzhoferetal10}. This package is designed for solving ODE systems with regular singular points.

Figure 1 shows that the value function increases in the correlation of the assets' fundamental price processes. This is natural as a negative correlation reduces risk costs. We also see that for our choice of model parameters the more liquid asset is liquidated at a much faster rate than the less liquid one and that the initial liquidation rate increases in the correlation. Both results are intuitive; fast liquidation reduces risk cost and the cost savings are increasing in the correlation. Moreover, the less liquid asset is liquidated at an almost constant rate while the more liquid asset is liquidated at a convex rate with the degree of convexity decreasing in the correlation. 

The convexity of the optimal liquidation strategy is consistent with the single asset case analyzed in \cite{GH2017}. There, it is shown that when the instantaneous market impact is small, the optimal liquidation strategy resembles a strategy with block trades: the (single) asset is liquidated at a very high rate initially and close to the terminal time. Similar results for the 2-dimensional case are shown in Figure 2 where the dependence of the value function (left) and the optimal liquidation strategy (right) on the impact factors $\lambda_1$ is depicted.

\begin{figure}[H]\label{figure1}
\begin{minipage}[c]{0.5\textwidth}
\centering
\includegraphics[height=4.5cm,width=7.5cm]{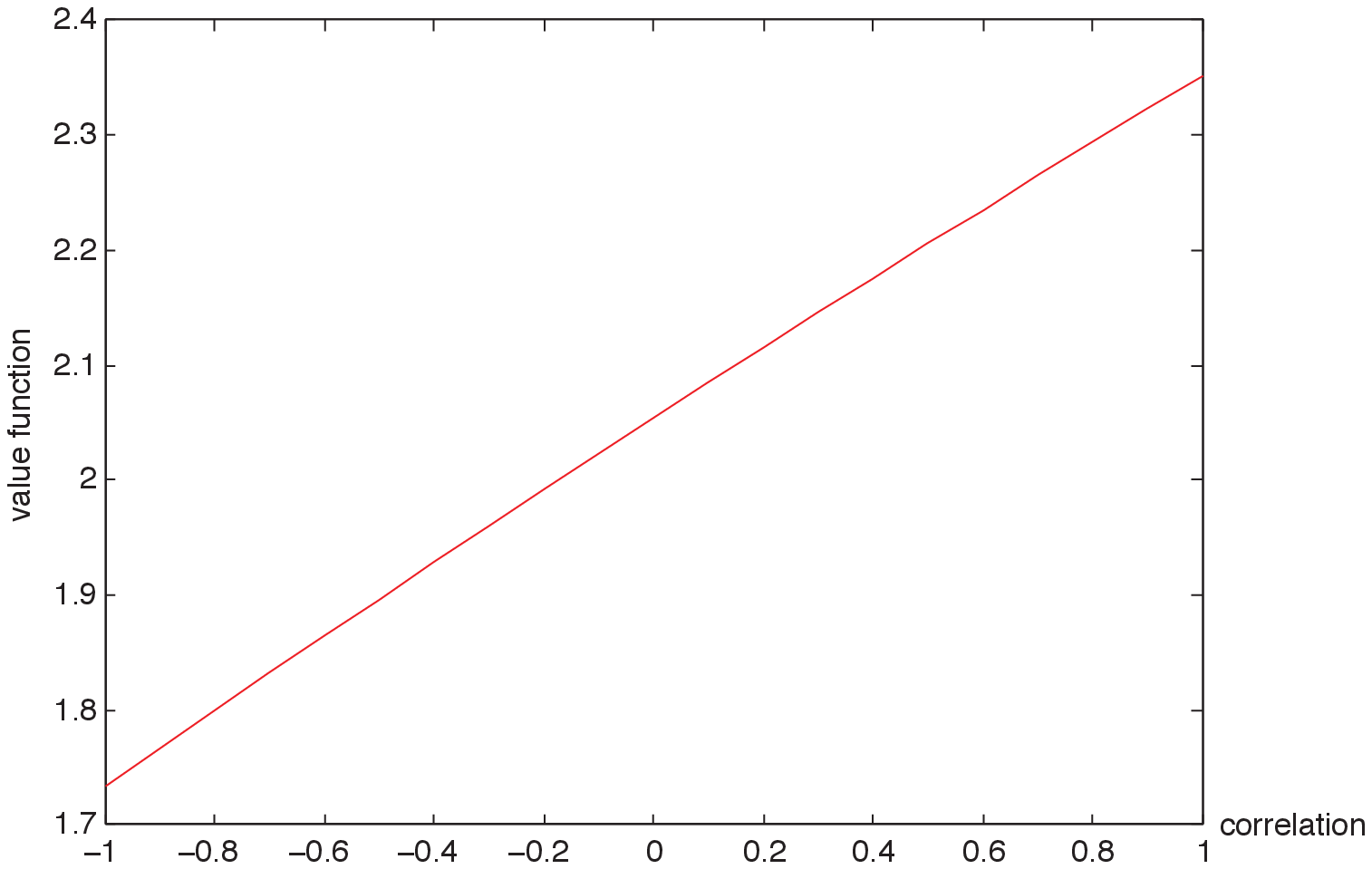}
\end{minipage}
\begin{minipage}[c]{0.5\textwidth}
\centering
\includegraphics[height=4.5cm,width=7.5cm]{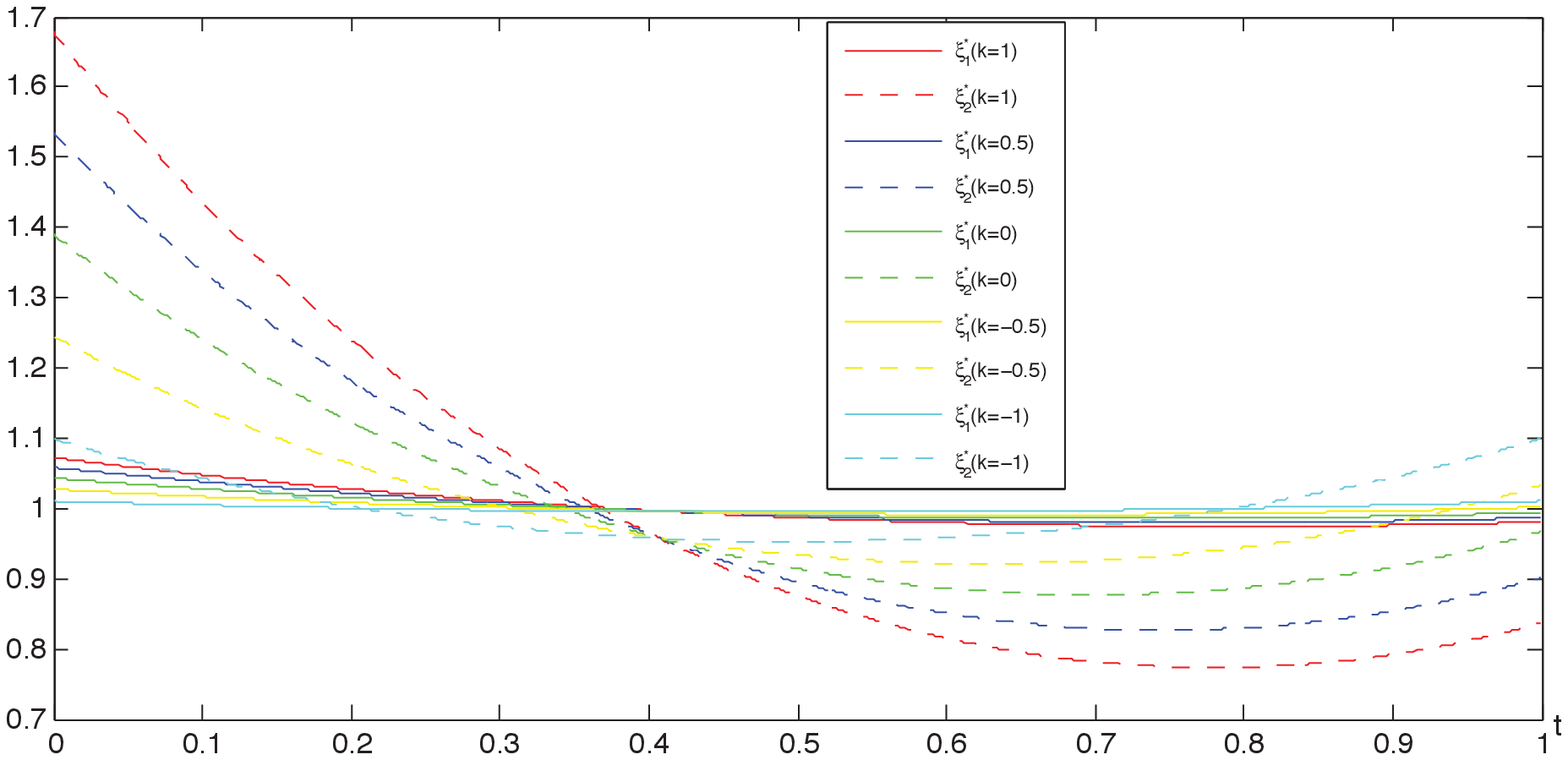}
\end{minipage}
\caption{Dependence of the value function (left) and the optimal trading
strategies (right) on the correlation $k$ for the parameter values $x_1 = x_2 = 1, y_1=y_2=0, T = 1, \lambda_1 = 10, \lambda_2 = 1, \sigma_1=\sigma_2=\gamma_1=\gamma_2=\rho_1=\rho_2=1$.}
\end{figure}

\begin{figure}[H]\label{figure2}
\begin{minipage}[c]{0.5\linewidth}
\centering
\includegraphics[height=4.5cm,width=7.5cm]{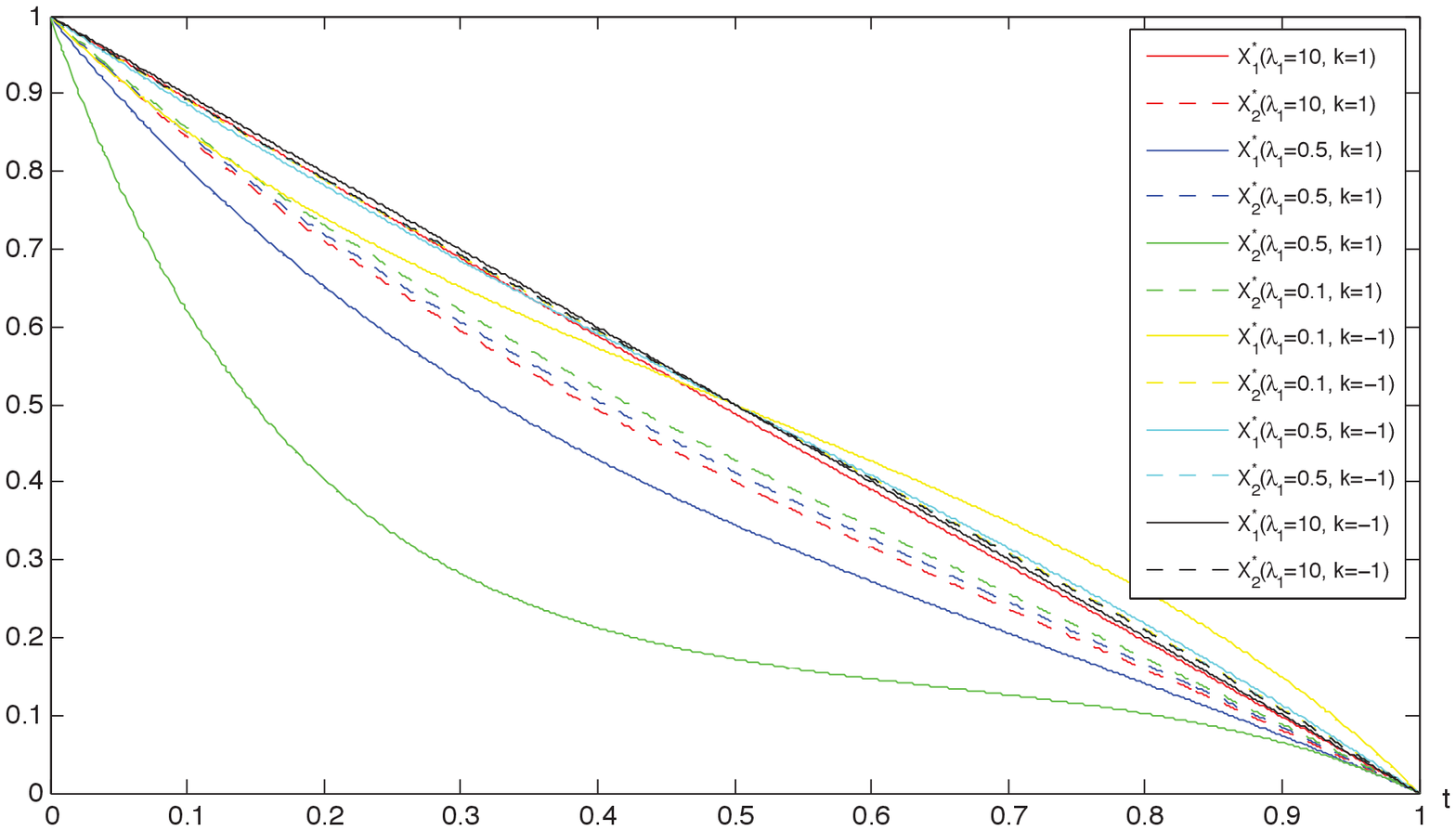}
\end{minipage}
\begin{minipage}[c]{0.5\linewidth}
\centering
\includegraphics[height=4.5cm,width=7.5cm]{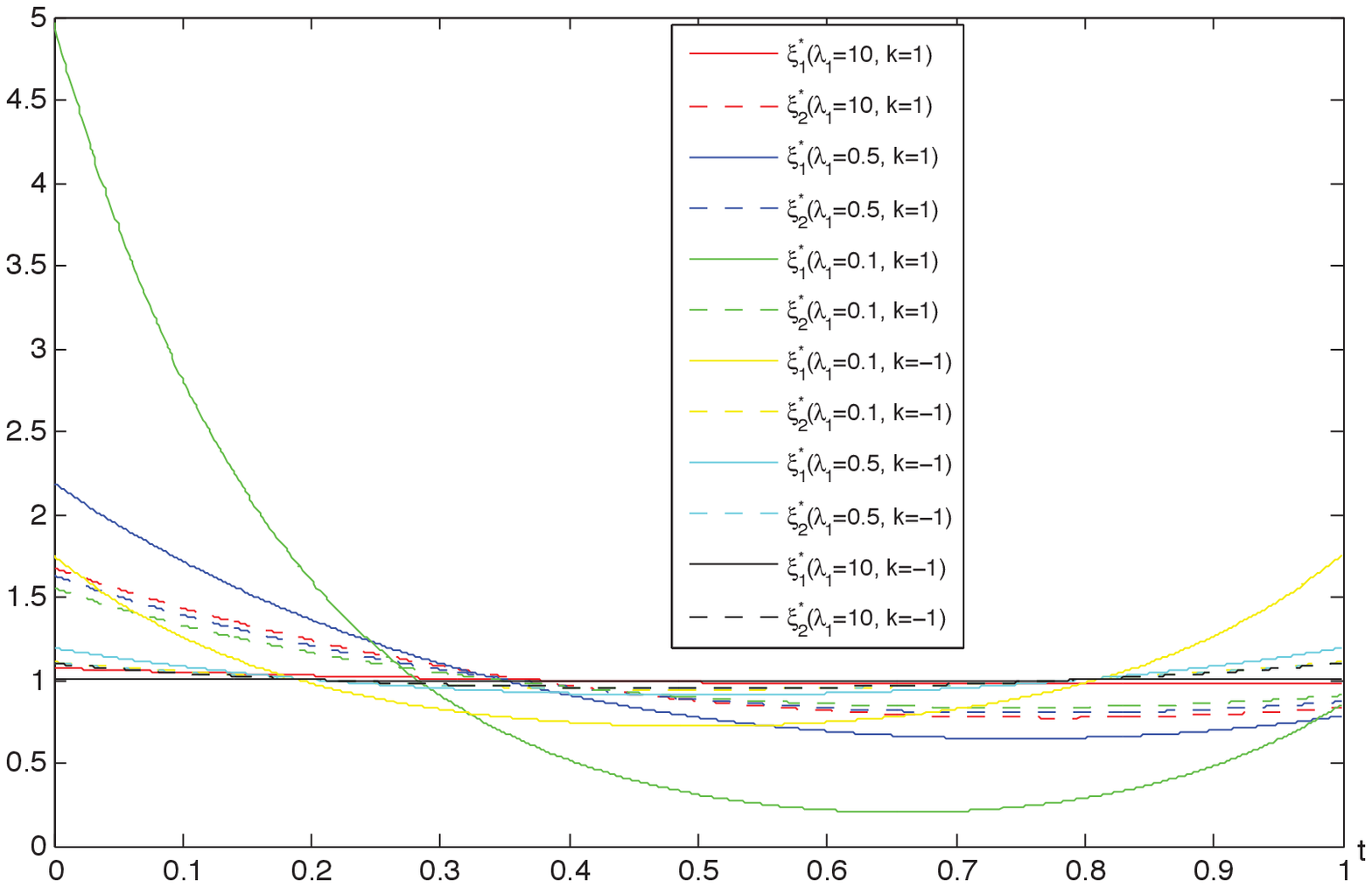}
\end{minipage}
\caption{Dependence of the optimal positions (left) and the trading strategies (right) on the instantaneous impact factor $\lambda_1$ for the parameter values $x_1 = x_2 = 1, y_1=y_2=0, T = 1, \lambda_2 = 1, \sigma_1=\sigma_2=\gamma_1=\gamma_2=\rho_1=\rho_2=1$. }
\end{figure}

While the initial trading rate {\sl decreases} if the instantaneous impact factor increases, our simulations suggest that it {\sl increases} with the persistent impact factors. This effect can already been seen in the single asset case as shown in Figure 3. Simulations for the 2-dimensional case are shown in Figure 4. If the persistent impact factor is large, early trading benefits from resilience. In fact, if there is no resilience and if the persistent impact dominates the cost function to the extend that we may drop the instantaneous impact and risk cost, the resulting Lagrange equation is zero and any liquidation strategy is optimal. If the resilience is positive large initial trades benefit from resilience effects.  

 The dependence of the optimal solution on the resilience factor $\rho_1$ is shown in Figure 5. Although we observe again that the optimal strategy is convex, the dependence of the convexity on the strength of resilience is less clear than that on the impact factors. We also see from that figure (red and green curves) that short positions can not be excluded; if the assets are strongly positively correlated, a negative position in one asset may well be beneficial in order to balance the portfolio risk.   

 \begin{figure}[H]\label{figure3}
\begin{minipage}[c]{0.5\linewidth}
\centering
\includegraphics[height=4.5cm,width=7.5cm]{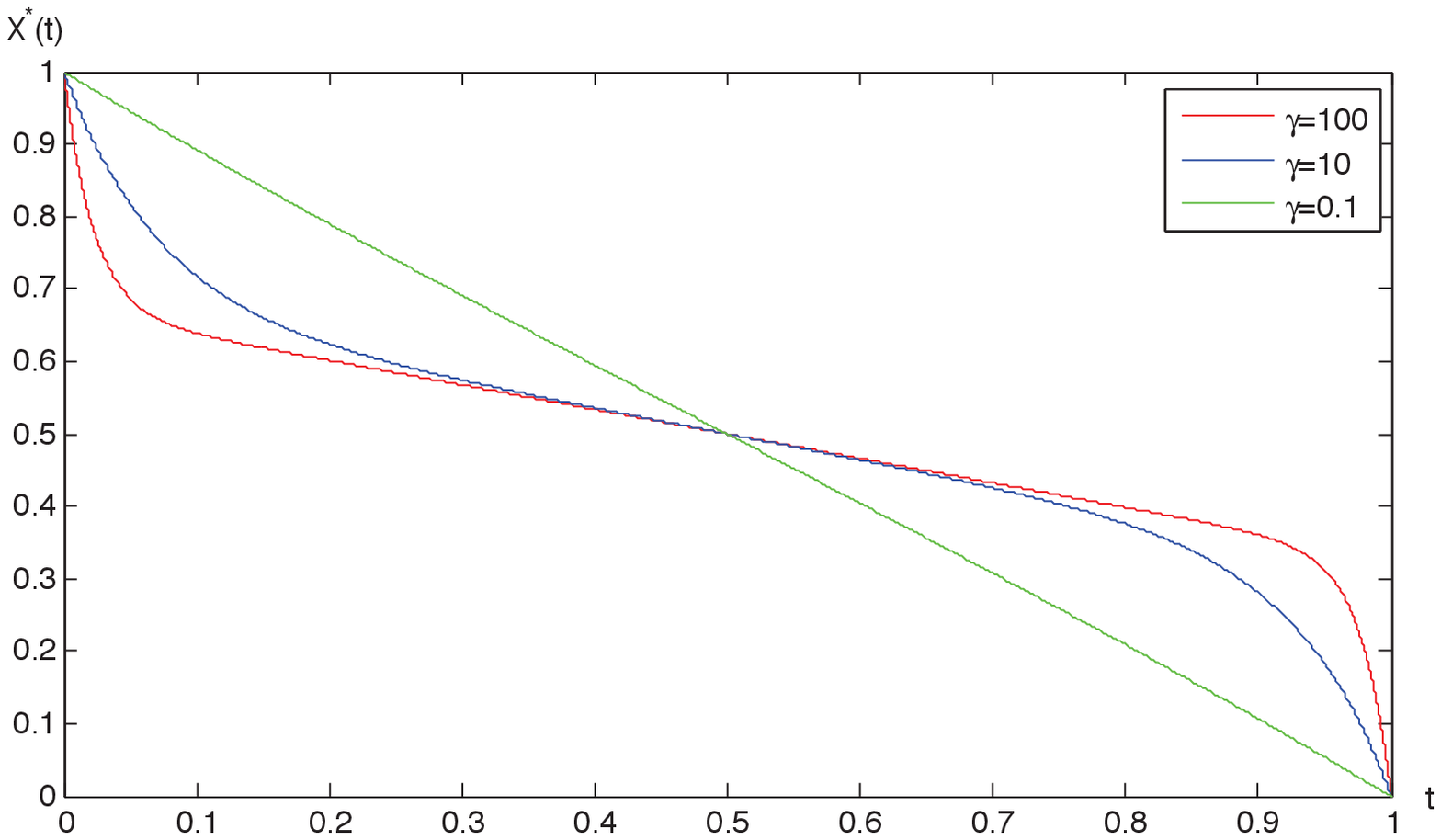}
\end{minipage}
\begin{minipage}[c]{0.5\linewidth}
\centering
\includegraphics[height=4.5cm,width=7.5cm]{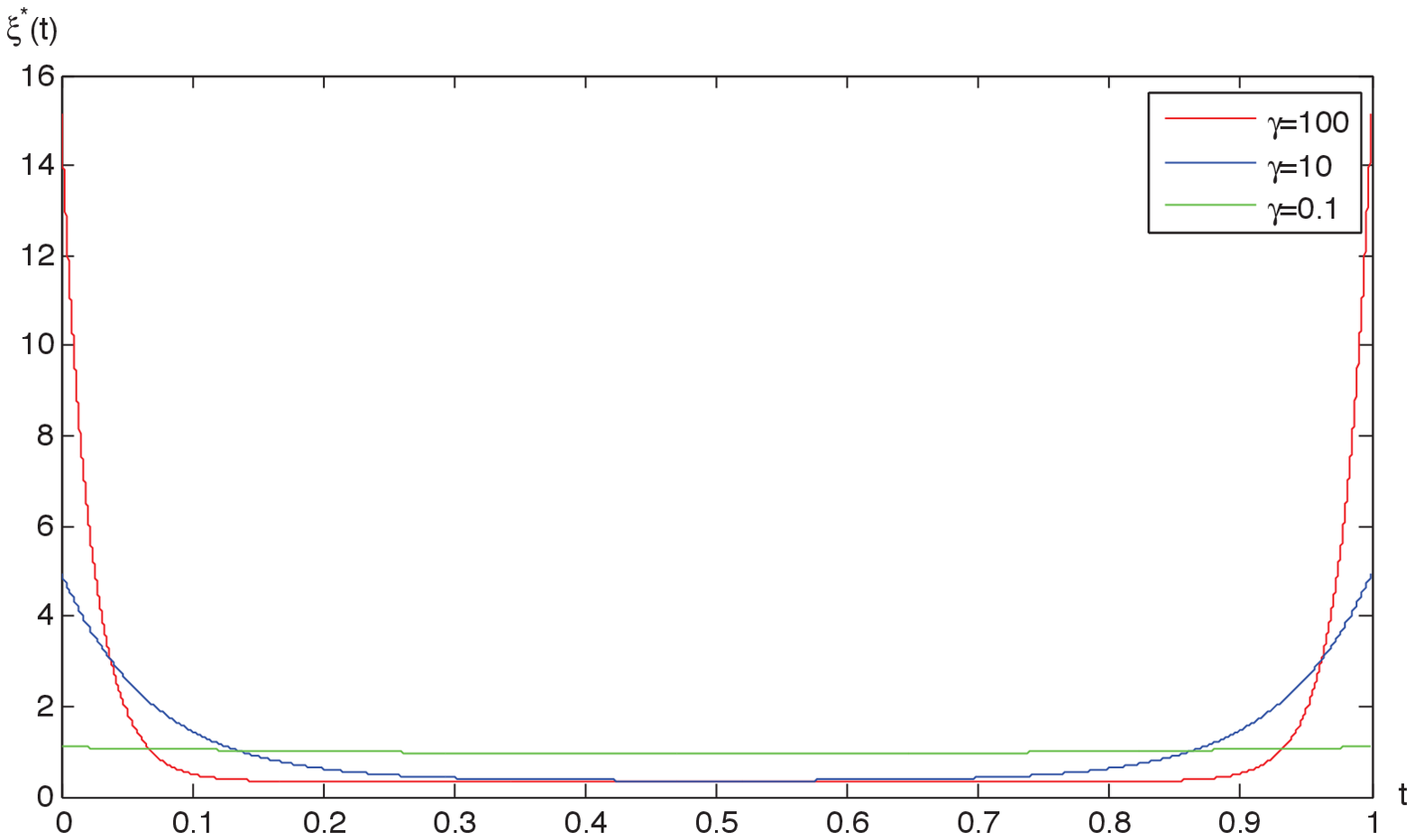}
\end{minipage}
\caption{Dependence of  the optimal position (left) and the trading strategy (right) in a single asset model on the persistent impact factor $\gamma$ for the parameter values $x= 1, y=0, T = 1, \lambda= 0.1, \sigma=0,  \rho=1$.}
\end{figure}

\begin{figure}[H]\label{figure4}
\begin{minipage}[c]{0.5\linewidth}
\centering
\includegraphics[height=4.5cm,width=7.5cm]{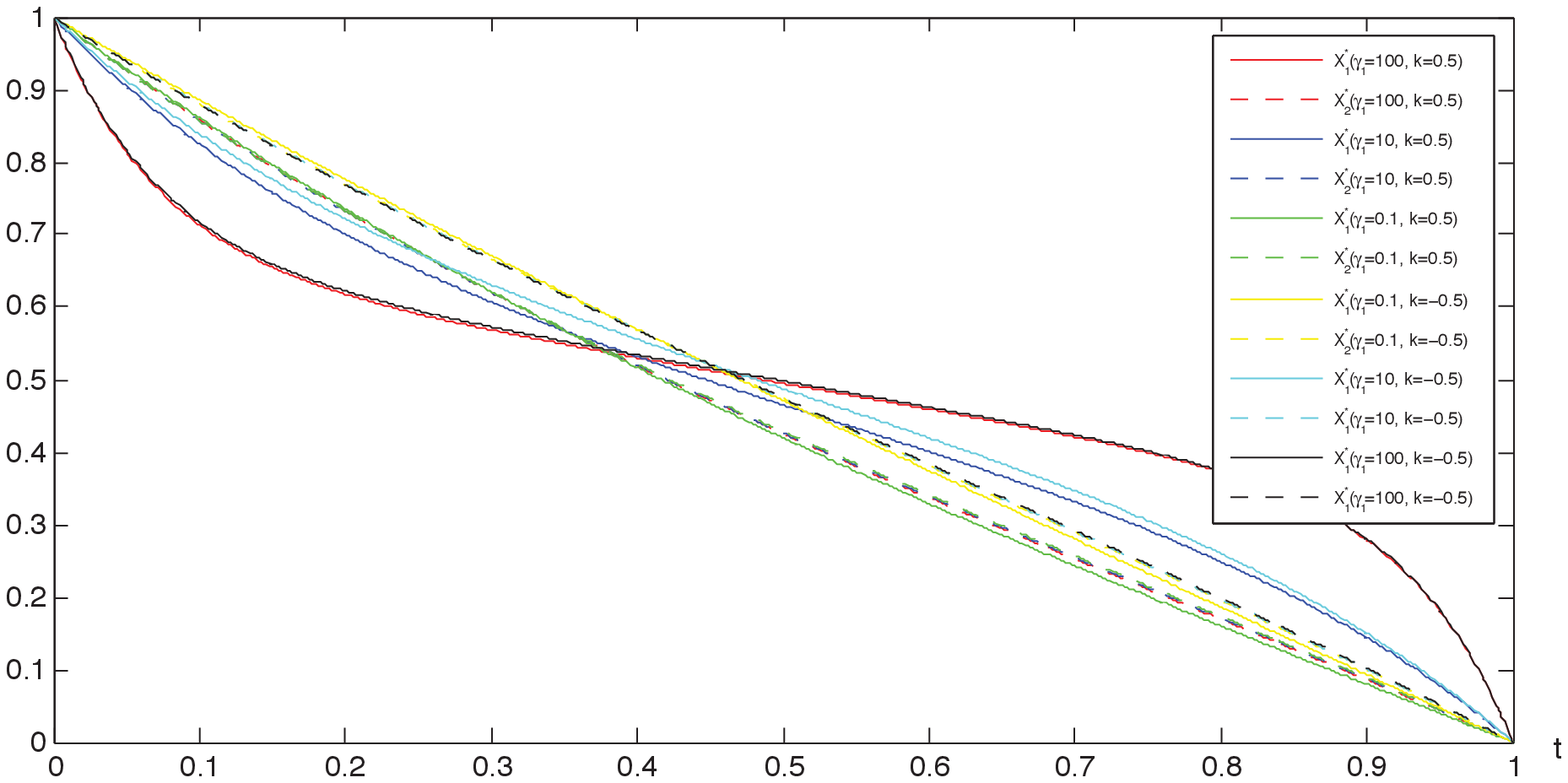}
\end{minipage}
\begin{minipage}[c]{0.5\linewidth}
\centering
\includegraphics[height=4.5cm,width=7.5cm]{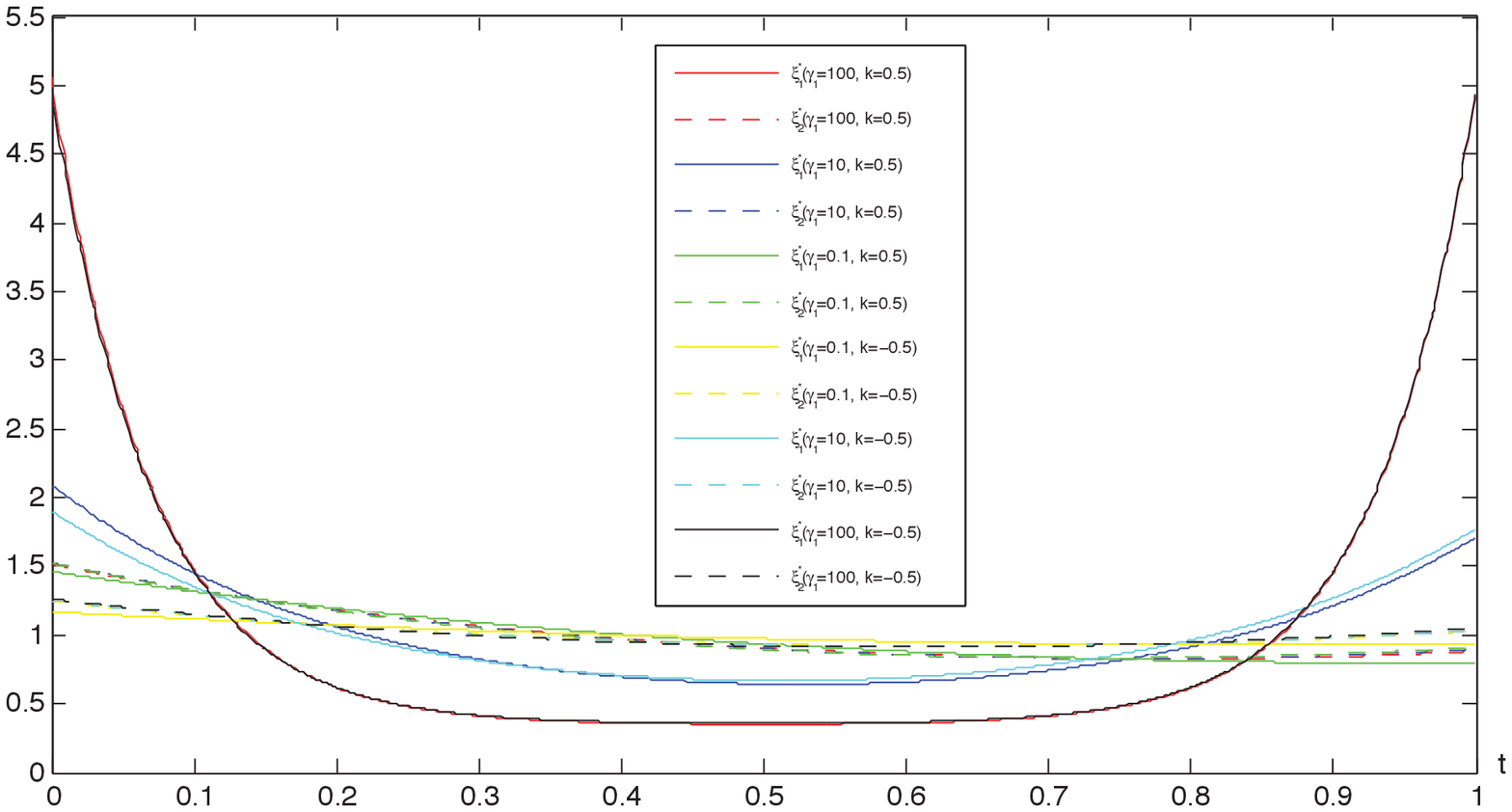}
\end{minipage}
\caption{Dependence of the optimal positions (left) and the trading strategies (right) on the persistent impact factor $\gamma_1$ for the parameter values $x_1 = x_2 = 1, y_1=y_2=0, T = 1, \lambda_1=\lambda_2 = 1, \sigma_1=\sigma_2=\gamma_2=\rho_1=\rho_2=1$.}
\end{figure}

\begin{figure}[H]\label{figure5}
\begin{minipage}[c]{0.5\linewidth}
\centering
\includegraphics[height=4.5cm,width=7.5cm]{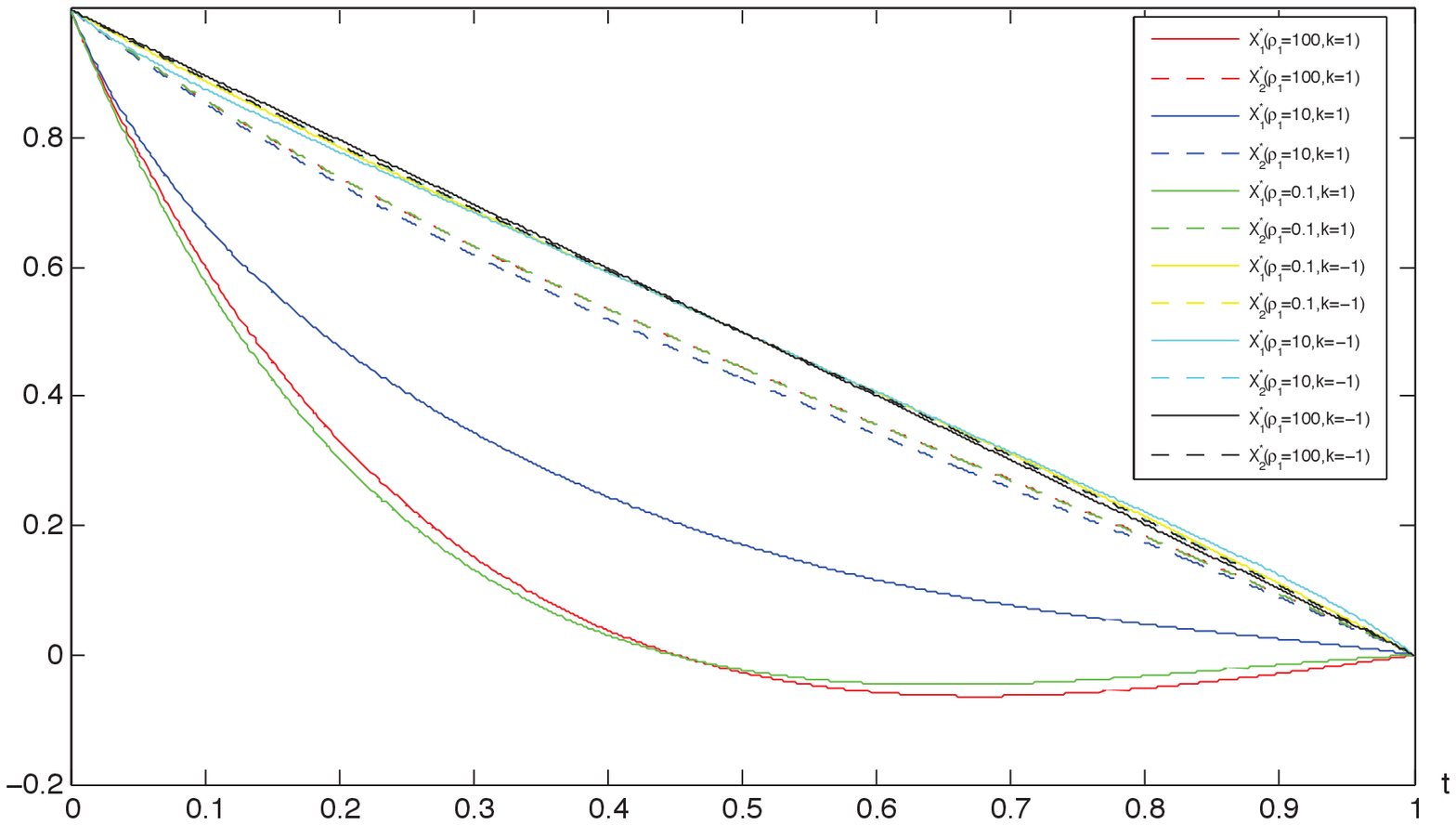}
\end{minipage}
\begin{minipage}[c]{0.5\linewidth}
\centering
\includegraphics[height=4.5cm,width=7.5cm]{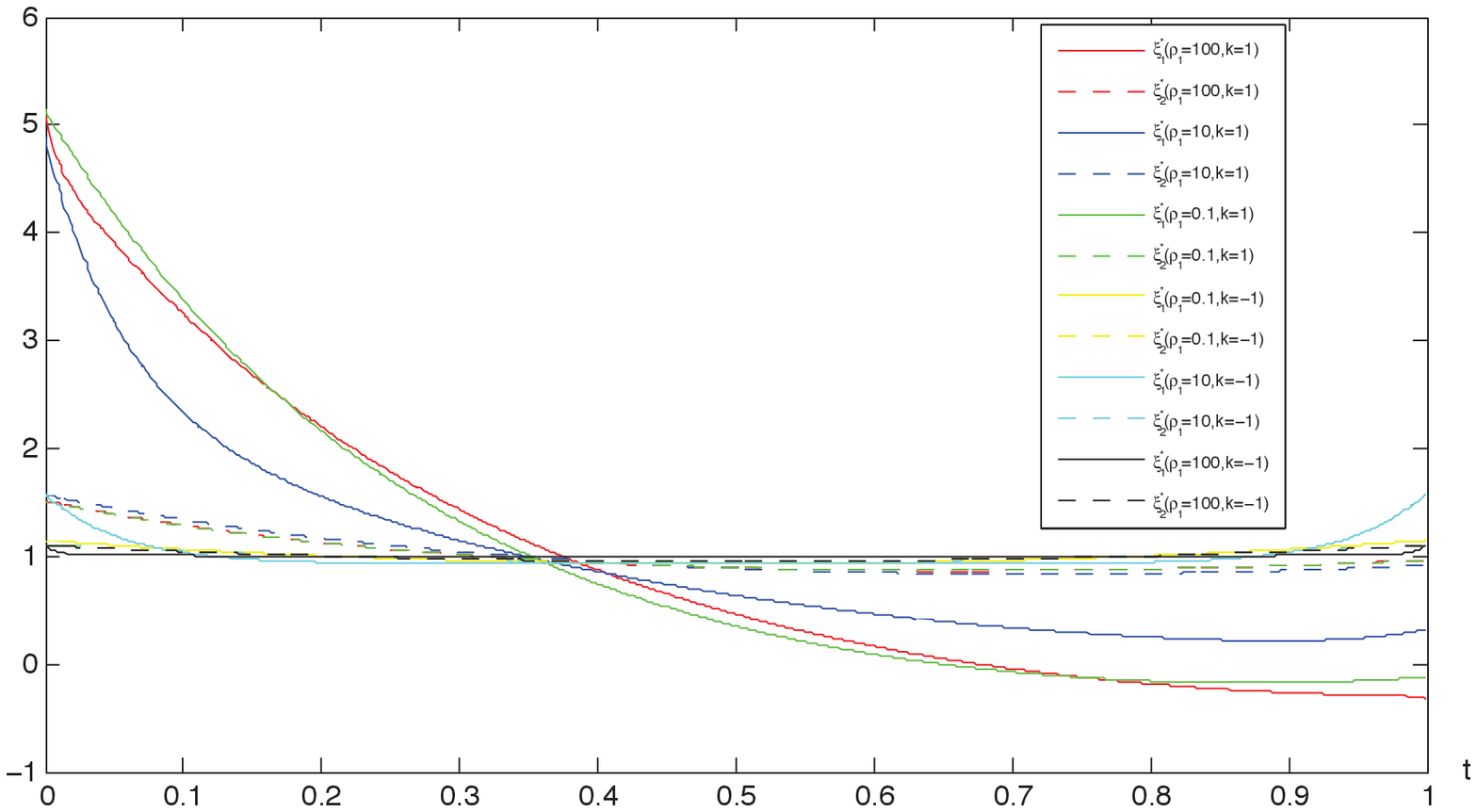}
\end{minipage}
	\caption{Dependence of the optimal positions (left) and the trading strategies (right) on the resilience factor $\rho_1$
 for the parameter values $x_1 = x_2 = 1, y_1=y_2=0, T = 1, \lambda_1=0.1,\lambda_2=\sigma_1=\sigma_2=\rho_2=\gamma_1=\gamma_2=1$.}
\end{figure}

\section{Conclusion}

In this paper we established the existence result of solutions  for a multi-dimensional stochastic control problem with singular terminal state constraint. In our model, the value function could be described by a matrix-valued LQ backward stochastic Riccati differential equation with singular terminal component. The verification argument strongly hinged on the a convergence result for the optimal strategies for a series of unconstrained control problems. Several avenues are open for further research. First, as in the one-dimensional case we cannot guarantee non-negativity of the trading rate. Second, the assumption that $\Lambda$ and $\gamma$ are constant was important to establish the a priori estimates. An extension to more general impact factors is certainly desirable. Finally, the set of admissible trading strategies was restricted to absolutely continuous ones. An extension to singular controls would be desirable as well.     

\appendix
\counterwithin{theorem}{section}

\section{Appendix}
\begin{theorem}\label{comparison}

For  $d\in  \mathbb N,$ let $G\in  L_{\mathcal F}^\infty(\Omega;C([0,T];\mathbb R^{d\times d}))$, $H_1,I_1,H_2,I_2\in  L_{\mathcal F}^\infty(\Omega;C([0,T];\mathcal S^{d}))$,  $S_1, S_2\in L_{\mathcal F}^\infty(\Omega; \mathcal S^{d}).$ Assume that  
$$S_1\leq S_2, \quad 0\leq H_{2}(\cdot)\leq H_{1}(\cdot), \quad I_{1}(\cdot)\leq I_{2}(\cdot) $$
on $[0,T].$ For $i=1,2$, let $(K_i,M_i)\in L_{\mathcal F}^\infty(\Omega;C([0,T];\mathcal S^{d}))\times L_{\mathcal F}^2(0,T;\mathcal S^{d})$ be the solution of the matrix-valued differential equation,
\begin{equation*} 
\begin{aligned}
-dK(t)&=-\{K(t)H_{i}(t)K(t)-G(t)^\mathrm TK(t)-K(t)G(t)-I_{i}(t)\}\,dt-M(t)\,dW(t),\\
K(T)&=S_i.
\end{aligned}
\end{equation*}
Then, 
$$	K_{1}(t)\leq K_{2}(t), \quad t\in [0,T].	$$
\end{theorem}
\begin{proof}
For given $(s,x) \in [0, T]\times \mathbb R^{d}$, let $y$ be the solution of
\begin{equation*}  
\left\{
\begin{aligned}
-dy(t)&=(G(t)-H_{2}(t)K_{2}(t))\,dt\\
y(s)&=x
\end{aligned}
\right.
\end{equation*}
on $[s,T].$ Then, 
\begin{equation*}  
\begin{aligned}
&\quad dy(t)^\mathrm T(K_{2}(t)-K_{1}(t))y(t)\\
&=y(t)^\mathrm T[(G(t)-H_{2}(t)K_{2}(t))^\mathrm T(K_{2}(t)-K_{1}(t))+(K_{2}(t)-K_{1}(t))(G(t)-H_{2}(t)K_{2}(t))]y(t)\,dt\\
&\quad +y(t)^\mathrm T(dK_{2}(t)-dK_{1}(t))y(t)\\
&=-y(t)^\mathrm T[(K_{2}(t)-K_{1}(t))H_{2}(t)(K_{2}(t)-K_{1}(t))+K_{1}(t)(H_{1}(t)-H_{2}(t))K_{1}(t)\\
&\quad +I_{2}(t)-I_{1}(t)]y(t)\,dt-y(t)^\mathrm T(M_{1}(t)-M_{2}(t))y(t)\,dW(t).
\end{aligned}
\end{equation*}
Thus,
\begin{equation*}  
\begin{aligned}
x^\mathrm T(K_{2}(t)-K_{1}(t))x=\mathbb E^{\mathcal F_s}\{&y(T)^\mathrm T(S_2-S_1)y(T)+\int_s^T y(t)^\mathrm T[K_{1}(t)(H_{1}(t)-H_{2}(t))K_{1}(t)\\
&+(K_{2}(t)-K_{1}(t))H_{2}(t)(K_{2}(t)-K_{1}(t))+I_{2}(t)-I_{1}(t)]y(t)\,dt\}\geq 0.
\end{aligned}
\end{equation*}
\end{proof}

\begin{lemma}\label{appendix-estimate}
Let $n>n_0$ for $n_0$ as in \eqref{n_0}. For fixed $t\in[0,T)$, there exists a constant $L$ independent of $n, s$ such that 
\begin{equation} \label{tilde estimate}
\begin{aligned}
e^{\int^s_t p^n(u)du}&\leq L\mathbb I_{s\in[0,T_0)}+\frac{L}{T-s+\frac{\lambda_{\min}}{n-\gamma_{\min}-\lambda_{\min}(1+\sqrt{1+\alpha})}}\mathbb I_{s\in[T_0,T]};\\
e^{\int^s_t- q^n(u)du}&\leq L\mathbb I_{s\in[0,T_0)}+L[T-s+\frac{\lambda_{\max}}{n-\gamma_{\max}+\lambda_{\max}}]\mathbb I_{s\in[T_0,T]},
\end{aligned}
\end{equation}
with
 $$\quad \alpha=\frac{||\Sigma||_{L^{\infty}}+2\gamma_{\max}||\rho||_{L^{\infty}}}{\lambda_{\min}}.$$
\end{lemma}
\begin{proof}
In the following proof we introduce simpler bounds $\tilde p^n ,\tilde q^n $ for $ p^n , q^n $ to simplfy the calculations. For $n>n_0,$ $\tilde q^n\leq q^n,  p^n\leq \tilde p^n$  where $\tilde q^n, \tilde p^n$are given by
 \begin{equation*}
\begin{aligned}
\tilde p^n(t)&=\left\{
\begin{aligned}
& p^n(t), \qquad \qquad \qquad \qquad \qquad \quad t\in[0,T_0);\\
&\frac{1}{T-t+\frac{1}{\frac{n-\gamma_{\min}}{\lambda_{\min}}-\sqrt{1+\alpha}-1}}+1, \quad t\in[T_0,T].
\end{aligned}\right.\\
\tilde q^n(t)&=\left\{
\begin{aligned}
& p^n(t), \qquad \qquad \qquad \qquad \qquad \quad t\in[0,T_0);\\
&\frac{1}{T-t+\frac{1}{\frac{n-\gamma_{\max}}{\lambda_{\max}}+1}}-1,\qquad \quad\   t\in[T_0,T].
\end{aligned}\right.\\
\end{aligned}
\end{equation*}
 Hence we need to prove that, 
\begin{eqnarray*} 
 e^{\int^s_t\tilde p^n(u)du} & \leq & L \mathbb I_{s\in[0,T_0)}+\frac{L}{T-s+\frac{\lambda_{\min}}{n-\gamma_{\min}-\lambda_{\min}(1+\sqrt{1+\alpha})}}\mathbb I_{s\in[T_0,T]};\\
e^{\int^s_t-\tilde q^n(u)du} & \leq & L  \mathbb I_{s\in[0,T_0)}+L[T-s+\frac{\lambda_{\max}}{n-\gamma_{\max}+\lambda_{\max}}]\mathbb I_{s\in[T_0,T]}.
\end{eqnarray*}
For $0\leq t< s <T_0,$
\begin{equation*} 
\begin{aligned}
e^{\int^s_t\tilde p^n(u)du}=e^{\tilde p^n(0)(s-t)}\leq e^{\tilde p^n(0)T_0}; \quad e^{\int^s_t-\tilde q^n(u)du}=e^{-\tilde q^{n}(0)(s-t)}\leq 1.
\end{aligned}
\end{equation*}
For $0\leq t<T_0\leq s\leq T ,$
\begin{equation*} 
\begin{aligned}
e^{\int^s_t\tilde p^n(u)du} & = e^{\tilde p^n(0)(T_0-t)}e^{(s-T_0)}\left[\frac{T-T_0+\frac{\lambda_{\min}}{n-\gamma_{\min}-\lambda_{\min}(1+\sqrt{1+\alpha})}}{T-s+\frac{\lambda_{\min}}{n-\gamma_{\min}-\lambda_{\min}(1+\sqrt{1+\alpha})}}\right]\\
& \leq e^{\tilde p^n(0)T_0}e^T\left[\frac{T-T_0+\frac{\lambda_{\min}}{n_2	-\gamma_{\min}-\lambda_{\min}(1+\sqrt{1+\alpha})}}{T-s+\frac{\lambda_{\min}}{n-\gamma_{\min}-\lambda_{\min}(1+\sqrt{1+\alpha})}}\right]
\end{aligned}
\end{equation*}
and
\begin{equation*} 
\begin{aligned}
e^{\int^s_t-\tilde q^n(u)du} & =e^{-\tilde q^{n}(0)(T_0-t)}e^{(s-T_0)}\left[\frac{T-s+\frac{\lambda_{\max}}{n-\gamma_{\max}+\lambda_{\max}}}{T-T_0+\frac{\lambda_{\max}}{n-\gamma_{\max}+\lambda_{\max}}}\right]\\
& \leq e^T\left[\frac{T-s+\frac{\lambda_{\max}}{n-\gamma_{\max}+\lambda_{\max}}}{T-T_0}\right].
\end{aligned}
\end{equation*}
For $T_0\leq t< s \leq T ,$
\begin{equation*} 
\begin{aligned}
e^{\int^s_t\tilde p^n(u)du} & = e^{(s-t)}\left[\frac{T-t+\frac{\lambda_{\min}}{n-\gamma_{\min}-\lambda_{\min}(1+\sqrt{1+\alpha})}}{T-s+\frac{\lambda_{\min}}{n-\gamma_{\min}-\lambda_{\min}(1+\sqrt{1+\alpha})}}\right]\\
& \leq e^T\left[\frac{T-t+\frac{\lambda_{\min}}{n_2	-\gamma_{\min}-\lambda_{\min}(1+\sqrt{1+\alpha})}}{T-s+\frac{\lambda_{\min}}{n-\gamma_{\min}-\lambda_{\min}(1+\sqrt{1+\alpha})}}\right];
\end{aligned}
\end{equation*}
and
\begin{equation*} 
\begin{aligned}
e^{\int^s_t-\tilde q^n(u)du} & =e^{(s-t)}\left[\frac{T-s+\frac{\lambda_{\max}}{n-\gamma_{\max}+\lambda_{\max}}}{T-t+\frac{\lambda_{\max}}{n-\gamma_{\max}+\lambda_{\max}}}\right]\\
& \leq e^T\left[\frac{T-s+\frac{\lambda_{\max}}{n-\gamma_{\max}+\lambda_{\max}}}{T-t}\right].
\end{aligned}
\end{equation*}
Therefore, for fixed $t\in[0,T)$, there exists a constant $L$ independent of $n, s$ such that 
\begin{equation*} 
\begin{aligned}
e^{\int^s_t p^n(u)du}&\leq e^{\int^s_t\tilde p^n(u)du}\leq L \mathbb I_{s\in[0,T_0)}+\frac{L}{T-s+\frac{\lambda_{\min}}{n-\gamma_{\min}-\lambda_{\min}(1+\sqrt{1+\alpha})}}\mathbb I_{s\in[T_0,T]};\\
e^{\int^s_t-q^n(u)du}&\leq e^{\int^s_t-\tilde q^n(u)du}\leq L \mathbb I_{s\in[0,T_0)}+L[T-s+\frac{\lambda_{\max}}{n-\gamma_{\max}+\lambda_{\max}}]\mathbb I_{s\in[T_0,T]}.
\end{aligned}
\end{equation*}
\end{proof}
\bibliographystyle{siam}
{
\bibliography{bib_diss,solution}

\begin{thebibliography}{10}

\bibitem{AlfonsiKloeckSchied16}
{\sc A.~Alfonsi, F.~Klöck, and A.~Schied}, {\em Multivariate transient price
  impact and matrix-valued positive definite functions}, Math. Oper. Res., 41
  (2016), pp.~914--934.

\bibitem{AlmgrenChriss00}
{\sc R.~Almgren and N.~Chriss}, {\em Optimal execution of portfolio
  transactions}, J. Risk, 3 (2001), pp.~5--39.

\bibitem{AnkirchnerJeanblancKruse14}
{\sc S.~Ankirchner, M.~Jeanblanc, and T.~Kruse}, {\em {BSDEs with singular
  terminal condition and control problems with constraints}}, SIAM J. Control
  Optim., 52 (2014), pp.~893--913.

\bibitem{Bernstein2005}
{\sc D.~S. Bernstein}, {\em Matrix Mathematics: Theory, Facts, and Formulas
  with Application to Linear Systems Theory}, Princeton University Press, 2005.

\bibitem{BertsimasLo98}
{\sc D.~Bertsimas and A.~W. Lo}, {\em Optimal control of execution costs}, J.
  Financ. Markets, 1 (1998), pp.~1--50.

\bibitem{Bismut1976}
{\sc J.-M. Bismut}, {\em Linear quadratic optimal stochastic control with
  random coefficients}, SIAM J. Control Optim., 14 (1976), pp.~419--444.

\bibitem{FruthSchoenebornUrusov14}
{\sc A.~Fruth, T.~Sch\"{o}neborn, and M.~Urusov}, {\em Optimal trade execution
  and price manipulation in order books with time-varying liquidity}, Math.
  Finance, 24 (2014), pp.~651--695.

\bibitem{FruthSchoenebornUrusov15}
\leavevmode\vrule height 2pt depth -1.6pt width 23pt, {\em Optimal trade
  execution in order books with stochastic liquidity}.
\newblock
  \url{http://homepage.alice.de/murusov/papers/fsu-optimal_execution_stochastic.pdf},
  2015.

\bibitem{GH2017}
{\sc P.~Graewe and U.~Horst}, {\em Optimal trade execution with instantaneous
  price impact and stochastic resilience}, SIAM J. Control Optim., 55 (2017),
  pp.~3707--3725.

\bibitem{GraeweHorstQiu13}
{\sc P.~Graewe, U.~Horst, and J.~Qiu}, {\em {A non-Markovian liquidation
  problem and backward SPDEs with singular terminal conditions}}, SIAM J.
  Control Optim., 53 (2015), pp.~690--711.

\bibitem{Kitzhoferetal10}
{\sc G.~Kitzhofer, O.~Koch, G.~Pulverer, C.~Simon, and E.~Weinm\"uller}, {\em
  {The new MATLAB code \textup{\texttt{bvpsuite}} for the solution of singular
  implicit BVPs}}, J. Numer. Anal. Indust. Appl. Math., 10 (2010),
  pp.~113--134.

\bibitem{Kohlmann2003}
{\sc M.~Kohlmann and S.~Tang}, {\em Multidimensional backward stochastic
  riccati equations and applications}, SIAM J. Control Optim., 41 (2003),
  pp.~1696--1721.

\bibitem{Kratz2011}
{\sc P.~Kratz}, {\em Optimal liquidation in dark pools in discrete and
  continuous time}, PhD thesis, Humboldt-Universität zu Berlin, 2011.

\bibitem{KratzSchoeneborn15}
{\sc P.~Kratz and T.~Sch\"{o}neborn}, {\em Portfolio liquidation in dark pools
  in continuous time}, Math. Finance, 25 (2015), pp.~496--544.

\bibitem{KratzSchoeneborn12}
{\sc P.~Kratz and T.~Schöneborn}, {\em Optimal liquidation in dark pools},
  Quant. Finance, 14 (2014), pp.~1519--1539.

\bibitem{KrusePopier16}
{\sc T.~Kruse and A.~Popier}, {\em Minimal supersolutions for bsdes with
  singular terminal condition and application to optimal position targeting},
  Stochastic Process. Appl., 126 (2016), pp.~2554--2592.

\bibitem{ObizhaevaWang13}
{\sc A.~Obizhaeva and J.~Wang}, {\em Optimal trading strategy and supply/demand
  dynamics}, J. Financ. Markets, 15 (2013), pp.~1--31.

\bibitem{Peng1992}
{\sc S.~Peng}, {\em Stochastic hamilton{\textendash}jacobi{\textendash}bellman
  equations}, SIAM J. Control Optim., 30 (1992), pp.~284--304.

\bibitem{SchiedSchoenebornTehranchi10}
{\sc A.~Schied, T.~Sch\"{o}neborn, and M.~Tehranchi}, {\em {Optimal basket
  liquidation for CARA investors is deterministic}}, Appl. Math. Finance, 17
  (2010), pp.~471--489.

\bibitem{SchneiderLillo2017}
{\sc M.~Schneider and F.~Lillo}, {\em Cross-impact and no-dynamic-arbitrage},
  SSRN,  (2016).

\bibitem{Schoeneborn16}
{\sc T.~Schöneborn}, {\em Adaptive basket liquidation}, Finance Stoch., 20
  (2016), pp.~455--493.

\end{thebibliography}
}

\end{document}